\documentclass[10pt]{article}

\usepackage{amsmath,amsthm,amssymb,amsxtra, graphicx}
\usepackage{amsmath}
\usepackage{graphicx}
\usepackage{amsfonts}
\usepackage{amssymb}
\usepackage{color}
\newtheorem{theorem}{Theorem}[section]
\newtheorem{lemma}{Lemma}[section]

\newtheorem{proposition}{Proposition}[section]
\newtheorem{corollary}{Corollary}[section]
\newtheorem{remark}{Remark}[section]

\renewcommand{\thefootnote}{\fnsymbol{footnote}}

\begin{document}

\baselineskip=18pt

\date{}
\title{\textbf{Extrapolation in Weighted Classical and Grand Lorentz Spaces. Application to the Boundedness of Integral operators}}

\author{Vakhtang ~Kokilashvili and Alexander~Meskhi}

\maketitle
\begin{abstract}
\noindent We establish  weighted extrapolation  theorems in classical and grand Lorentz spaces. As a consequence we have the
weighted boundedness of operators of Harmonic Analysis in grand Lorentz spaces. We treat  both cases: diagonal and off-diagonal ones.
\end{abstract}

\renewcommand{\thefootnote}{\fnsymbol{footnote}}
\footnotetext{2010\emph{Mathematics Subject Classification}:
46E30;  42B20; 42B25.} 
 \footnotetext{\emph{Running head}: Extrapolation in  Lorentz Spaces}



\vspace{10 mm}

\section{Introduction}
Our aim is to introduce new weighted grand Lorentz spaces and to derive Rubio de Franc\'ia's weighted extrapolation results in these spaces.
The obtained  results are applied to get the boundedness of operators of Harmonic Analysis in weighted grand Lorentz spaces.
To derive the boundedness of operators we rely on  a weighted extrapolation theorem in the classical Lorentz spaces which has an independent interest.
To get the latter result we first prove   weighted extrapolation statements for Banach function spaces.  Rubio de Franc\'ia's  extrapolation  theory  gives  powerful tools  in the study mapping properties  of integral operators in weighted function spaces. One of the  important properties of the $A_p$  weights is the extrapolation theorem announced  by  Rubio de Franc\'ia \cite{Rubio},
and given with a detailed proof in \cite{Rubio1}. The first version of the extrapolation
theorem says that if for some $p_0$,  a sublinear operator is bounded in $L^{p_0}_w$ for
all $w\in   A_{p_0/\lambda}$ with $1\leq \lambda < \infty$ and $\lambda\leq p\leq \infty$,  then it is bounded in  $L^p_w$  for
all $w \in  A_{p/\lambda}$ and $\lambda< p< \infty$. There exists long list of papers which deals with different proofs of this theorem,
generally speaking,  in various function spaces and related topics (see e.g., \cite{Doan1}, \cite{CMP} and references cited therein).

\section{Preliminaries}\label{Prelim}

Let $(X, d, \mu)$ be a quasi-metric measure space with a quasi-metric $d$ and measure $\mu$. A~quasi-metric $d$
is a function $d\colon X \times X \rightarrow [0,\infty)$ which satisfies the following conditions:

\begin{itemize}
\item[{\rm(i)}] $d(x,y)=0$ if and only if  $x=y$;
\item[{\rm(ii)}] for all  $x,y\in X$, $d(x,y) =  d(y,x)$;
\item[{\rm(iii)}] there is a positive constant $\kappa$ such that $d(x,y)\leq \kappa\,(d(x,z)+d(z,y))$ for all $x,y,z\in X$.
\end{itemize}

In what follows we will assume that the balls $B(x, r) := \{y\in X; \, d(x, y)<r\}$ are measurable with positive $\mu$ measure for all $x\in X$ and $r>0$.

If $\mu$ satisfies the doubling condition, i.e., there is a positive constant $D_{\mu}$ such that for all $x\in X$ and $r>0$,

\begin{equation}\label{doubling}
\mu (B(x,2r)) \leq D_{\mu} \mu (B(x,r)),
\end{equation}
then we say that $(X, d, \mu)$ is a space of homogeneous type ($SHT$). Throughout the paper we will assume that $(X, d, \mu)$ is an $SHT$.


For the definition, examples and some properties of an $SHT$ see, e.g.,  the paper \cite{MS} and the  monographs \cite{StTo}, \cite{CoWe}.

Throughout the paper, when we deal with  an $SHT$, we will assume the class of continuous functions is dense in $L^1(X)$.

For a given quasi-metric measure  space $(X, d, \mu)$ and $q$ satisfying $1\leq q\leq \infty$, we will denote by $L^q=L^q(X,\mu)$ the Lebesgue
space equipped with the standard norm.


\noindent

Let $f$ be a $\mu-$ measurable function on $X$ and let $1\leq p < \infty$, $1\leq s \leq \infty$.
Suppose that $w$ is a weight function on $X$, i.e. $w$ is $\mu-$ a.e.
positive and locally integrable on $X$.
We say that $f$ belongs to the weighted Lorentz space $L^{p,s}_w(X)$ ($L^{p,s}_w$ shortly) if

$$
 \|f\|_{L^{p,s}_w}= \begin{cases}
 \bigg( s \int\limits_0^{\infty} \big(  w \{ x\in X:  |f(x)|> \tau \}\big)^{s/p} \tau^{s-1} d\tau  \bigg)^{1/s},
\;\; \text{if}\; 1\leq s< \infty,  \\
\sup_{s>0} s \Big( w( \{ x\in X:  |f(x)|> s \}  \Big)^{1/p} , \;\; \text{if} \; s= \infty
\end{cases}
$$
is finite, where
$$ w E := \int\limits_E w(x) d\mu(x). $$

It is easy to see that  $L^{p,p}_w(X)$ coincides with the weighted Lebesgue space $L^p_w$.

Denote by $f^{*}_w$ a weighted non-increasing rearrangement of $f$ with respect to the measure $d\nu= w d\mu$.
Then by integration by parts it can be checked that (see also \cite{Hu}):

$$
\|f\|_{L^{p,s}_w}=
\begin{cases}
\bigg(  \frac{s}{p} \int\limits_{0}^{\infty} \bigg(  t^{1/p} f^{*}_w (t)  \bigg)^{s} \frac{dt}{t} \bigg)^{1/s}, \;\; \text{if} \; 1\leq s< \infty,  \\
\sup_{t>0} \{ t^{1/p} f^*_w(t)\},\;\;   \text{if} \;  s= \infty,   \}
\end{cases}
$$

Now we list some useful properties of Lorentz spaces (see e.g., , \cite{Hu}, \cite{CR} (Ch. 6), \cite{KK}):

(i) If $1\leq  p< \infty$ and $1\leq s \leq   \infty$, then $L^{p,s}_w(X)$ is a Banach space with the norm

$$ \|f\|_{(p,s,w)}= \begin{cases} \bigg(\int\limits_0^{\infty}[t^{1/p}f^{**}_w(t)]^s\frac{dt}{t}\bigg)^{1/s}, \;\; 1\leq s ,\infty, \\
\sup_{t>0} tf^{**}_w(t), \;\;\;\; s= \infty  \end{cases}$$
which is equivalent to $\| \cdot \|_{L^{p,s}_w}$, where
$$  f^{**}_w(t)=\int\limits_0^t f^*_w(\tau)d\tau; $$

(ii) $\| \chi_{E} \|_{L^{p,s}_w} = (w E)^{1/p}$;

(iii) If $1\leq p <\infty$, $s_2 \leq s_1$, then $L^{p,s_2}_w  \hookrightarrow L^{p,s_1}_w$ with the embedding constant $C_{p,s_1, s_2}$ depending only on $p$, $s_1$ and $s_2$;

(iv) There is a positive constant $C_{p,s}$ such that
$$ C_{p,s}^{-1} \|f \|_{L^{p, s}_w} \leq \sup_{\|h\|_{L^{p',s'}_w}\leq 1} \bigg|  \int\limits_{X} f(x) h(x) w(x) d\mu(x)
\bigg| \leq C_{p,s} \|f \|_{L^{p, s}_w}$$
for every $f\in L^{p,s}_w$, where $p'= p/(p-1)$, $s'= s/(s-1)$.

(v) (H\"older's inequality) Let $\frac{1}{p}= \frac{1}{p_1} + \frac{1}{p_2}$, $\frac{1}{s}= \frac{1}{s_1} + \frac{1}{s_2} $. Then

$$ \|f_1  f_2\|_{L^{p,s}_w} \leq  C \|f_1  \|_{L^{p_1, s_1}_w} \|f_2  \|_{L^{p_2, s_2}_w} $$
for all $f\in  L^{p_1, s_1}_w$ and $ f_2 \in L^{p_2, s_2}_w$, where $C= C_{p,s, p_1, p_2, s_1, s_2}$;

(vi)  $$ \|f^{1/q_0} \|^{q_0}_{L^{p,s}_w} = \|f\|_{L^{p/q_0, s/q_0}_w} $$
for $f\in L^{p/q_0, s/q_0}_w$ with $q_0$ satisfying the condition $p/q_0>1$, $s/q_0> 1$.

\vskip+0.1cm
Taking property (iv) into account we have the following statement:

\begin{proposition}\label{dual-miu}    There is a  positive constant $C_{p,s}$ such that
\begin{equation}\label{dual}
C_{p,s}^{-1} \|f \|_{L^{p, s}_w} \leq \sup_{\|w^{-1}h\|_{L^{p',s'}_w}\leq 1} \bigg| \int\limits_{X} f(x) h(x) d\mu(x) \bigg| \leq C_{p,s} \|f \|_{L^{p, s}_w},
 \end{equation}
with the same constant $C_{p,s}$ as in $(iv)$. Hence, the K\"othe dual  space  $\Big( L^{p,s}_w\Big)'$  of $L^{p,s}_w$ with respect to the  measure   space $(X,  \mu)$ (not with respect to the measure space $(X, \nu)$, where $d\nu= w d\mu$) is given by the norm equivalent to the quasi-norm $  \| w^{-1} f\|_{L^{p',s'}_w(X)}.  $
\end{proposition}

\begin{remark}\label{mainremark} In the sequel constants of the type $C_{p,s,\cdots,}$ depending on parameters  $p,s, \cdots$ (for example,
on parameters  of Lorentz spaces) and having the property

\begin{equation}\label{const}
\sup_{0< \varepsilon, \eta, \cdots,  < \sigma_0} C_{p-\varepsilon, s- \eta, \cdots}  \equiv C < \infty;
\end{equation}
where $\sigma_0$ is a small positive constant, will be denoted by $C$.

For example, the constants from (iii), (iv), \eqref{dual} and (v) have such a property (see e.g., \cite{Hu}).

Condition \eqref{const} is satisfied, for example, if the mappings $(p,s, \cdots) \mapsto C_{p,s, \cdots}$  are continuous with respect to
$p,s, \cdots, $.
\end{remark}

Let  $X$ be bounded (i.e., it is contained in some ball) and let $w$ be a weight on $X$.   In this case, $w$ is integrable on  $X$.  By the definition the weighted Iwaniec-Sbordone space $L^{p), \theta}_w(X)$  is defined with respect to the norm:

$$ \| f\|_{L^{p), \theta}_w(X)}= \sup_{0< \varepsilon< p-1} \varepsilon^{\frac{\theta}{p-\varepsilon}} \|f\|_{L^{p- \varepsilon}_w(X)}, \;\;
1<p<\infty, \;\; \theta>0. $$

The space $L^{p), \theta}_w(X)$  for $w\equiv 1$ and $X=\Omega$, where $\Omega$ is a bounded domain in ${\mathbb{R}}$,  was introduced in \cite{IwSb} for $\theta=1$ and in \cite{GIS} for $\theta>0$.

For structural properties of $L^{p), \theta}_w(X)$ spaces and mapping properties of operators of Harmonic Analysis in these spaces we refer,
e.g., to the monograph \cite{KMRS2} and references cited therein.

In the paper \cite{MeJMS} (see also \cite{KMRS2}, Ch. 14) it was introduced the grand Lorentz space on the  interval $(0,1)$ as follows: we say that $f\in \Lambda^{p), \theta}_w$,
$0< p< \infty$,  if
$$ \| f\|_{ \Lambda^{p), \theta}_w}=  \sup_{0 < \varepsilon < \varepsilon_0} \bigg(\varepsilon^{\theta} \int\limits_0^1 |f^{*}(t)|^{p-\varepsilon}
w(t) dt \bigg)^{\frac{1}{p-\varepsilon}} < \infty, $$
where $f^{*}$ is the decreasing rearrangement of $f$ with respect to the Lebesgue measure on $(0,1)$.
In the same paper the boundedness of the Hardy--Littlewood maximal operator was established in $\Lambda^{p), \theta}_w$ (see also \cite{JK}, \cite{FiKaZAA} for related topics).
Here $g^{*}$ is decreasing rearrangement of $g$ with respect to Lebesgue measure and $\varepsilon_0$ is defined as follows:

$$ \varepsilon_0=
\begin{cases} p-1, \;\; \text{if} \; p>1, \\
p, \; \text{if}\; p \leq 1.
\end{cases}
$$

Now we introduce grand Lorentz space in a different way. In particular, for a measurable function $f$ and a weight function $w$ on $X$, we let

$$  \|f\|_{  L^{p), s, \theta}_w }= \sup_{0< \varepsilon < p-1} \varepsilon^{\frac{\theta}{p-\varepsilon}}
 \| f \|_{  L^{p-\varepsilon, s}_w  }, $$
where $1< p< \infty$, $1\leq s \leq \infty$.


Let $1<p< \infty$. A weight function $w$ defined on $X$  belongs to the Muckenhoupt class $A_p(X)$ if
\begin{eqnarray*}
[w]_{A_p(X)} :=  \sup_{B} \bigg( \frac{1}{\mu(B)} \int\limits_{B}  w(x) \, d\mu(x) \bigg)
\bigg( \frac{1}{\mu(B)} \int\limits_{B} w^{1-p'}(x) d\mu(x)\bigg)^{p-1} <\infty,
\end{eqnarray*}
where the supremum is taken over all balls $B \subset X$.

Further, we say that $w\in A_1(X)$ if

\begin{equation}\label{a1}
(M w) (x)  \leq C w(x), \;\;\; for \;\; \mu- a.e. \;  x,
\end{equation}
where $M$ is the Hardy--Littlewood maximal operator defined on  $X$, i.e.,

$$ M g(x)= \sup_{B \ni x} \frac{1}{ \mu (B)}  \int\limits_{B} |g(y)| \; d\mu(y). $$

We denote by $[w]_{A_1}$ the best possible constant in (\ref{a1}).

The class of weights $A_{\infty}$ is the union of classes $A_p$, $1\leq p< \infty$. Further (see \cite{Hru} and  \cite{HuPeRe}),
$$ [w]_{A_{\infty}} :=  \sup_B \bigg( \frac{1}{\mu(B)} w d\mu \bigg) \exp \bigg( \frac{1}{\mu(B)} \int\limits_B \log w^{-1}d\mu\bigg). $$

There exists also another $A_{\infty}$ characteristic  due to \cite{F}:
$$ [w]_{A_{\infty}}^W:=  \sup_{B} \frac{1}{w(B)} \int\limits_B M(w \chi_B)d\mu.  $$
It can be checked (see also \cite{HuPeRe}) that
$$[w]^{W}_{A_{\infty}} \leq C_{\kappa, \mu} [w]_{A_{\infty}} \leq \overline{C}_{\kappa, \mu} [w]_{A_{p}}$$
with some structural constants $C_{\kappa, \mu}$ and $\overline{C}_{\kappa, \mu}$.

Let $1<p,q<\infty$. Suppose that $\rho$ is $\mu$-a.e. positive function such that $\rho^q$ is locally integrable. We say that
$\rho\in {\mathcal{A}}_{p,q}(X)$ if

\begin{eqnarray*}
[\rho]_{{\mathcal{A}}_{p,q}}:= \sup_{B} \bigg( \frac{1}{\mu B} \int\limits_{B} \rho^q \; d\mu \bigg)
\bigg( \frac{1}{\mu B} \int\limits_B \rho^{-p'}\, d\mu \bigg)^{q/p'}< \infty,
\end{eqnarray*}
where the supremum is taken over all  balls  $B \in  X$.

If $p=q$, then we denote ${\mathcal{A}}_{p,q}$ by ${\mathcal{A}}_{p}$. The next relation can be checked immediately
\begin{eqnarray*}
[\rho]_{{\mathcal{A}}_{p,q}}= [ \rho^q ]_{A_{1+ q/p'}}, \;\; 1<p\leq q< \infty.
\end{eqnarray*}
In particular, this equality  for $p=q$ has the form
\begin{eqnarray*}\label{equality1}
[\rho]_{{\mathcal{A}}_{p}}= [ \rho^p ]_{A_{p}}, \;\; 1<p< \infty.
\end{eqnarray*}

Since the Lebesgue differentiation theorem holds in $(X,d,\mu)$, it can be checked that
$$ [w]_{A_p}\geq 1; \;\;\; [\rho]_{{\mathcal{A}}_{p,q}} \geq 1. $$

\vskip+0.2cm

Due to H\"older's inequality the following monotonicity property of $A_p$ classes  holds:
\begin{equation}\label{monotone}
 [w]_{A_q} \leq [w]_{A_p}, \;\;\;\; 1\leq p\leq q< \infty.
 \end{equation}

It can be also verified that
\begin{equation}\label{1-p'}
\Big[  w\Big]_{A_p(X)}= \Big[  w^{1-p'}\Big]^{p-1}_{A_{p'}(X)}.
\end{equation}

Further, let $1< p< \infty$ and $1 \leq  s\leq \infty$. We say that a weight function $w$ belongs to the class $A(p,s)$ if there is a positive constant $C$ such that
$$ \| \chi_B\|_{L^{p,s}_w(X)} \| w^{-1} \chi_B \|_{L^{p', s'}_w(X)} \leq C \mu(B). $$

The class of weights $A(p,s)$ was introduced in \cite{ChHuKu} in Euclidean spaces. In the same paper (see also \cite{GeKo}) it was shown that
$w\in A_{p,s}$ if and only if $w\in A_p$ provided that $1< s\leq \infty$.

Let us recall that  the following Buckley-type estimate holds for the Hardy-Littlewood maximal operator $M$ defined on an $SHT$:

\begin{equation}\label{HuPeRe}
\| M\|_{L^{p}_w(X) \to L^{p}_w(X)}\leq \overline{c} p' [w]_{A_p(X)}^{1/(p-1)}, \;\; 1<p<\infty,
\end{equation}
where $[w]_{A_p(X)}$ is the $A_p$ characteristic of a weight $w$ defined on $X$ (see \cite{HuPeRe}).  For example, if $X$ is an interval
in ${\mathbb{R}}$, then we can take  $\overline{c}=2$.  According to  \cite{HuPeRe} the constant $\overline{c}$ is defined
as follows (see also \cite{KoMePositivity})
\begin{equation}\label{overlinesmall}
\overline{c}= 32 \kappa^{D_{\mu}}(2\theta)^{D_{\mu}}(1+ \tau_{\kappa, \mu}),
\end{equation}
where
\begin{equation}\label{tau}
\tau_{\kappa, \mu}= 6 \big(32\kappa^4(4\kappa+1)\big)^{D_{\mu}},
\end{equation}
$\theta= 4\kappa^2+\kappa$, $D_{\mu}$ is the constant defined by (\ref{doubling}),
$\kappa$ is the triangle inequality constant for the quasi-metric $d$.

In fact, in  \cite{HuPeRe} the authors established more general bound for
$\| M\|_{L^{p}_w(X)} $ involving $A_{\infty}$ characteristic but for our aims it suffices to apply  estimate \eqref{HuPeRe}.

In what follows we use standard notation from Banach space theory and operator theory. Let $L^0(\mu)= L^0(X, \mu)$ be the space of (equivalence classes of) $\mu$-measurable real-valued functions. A Banach space $E$ is said to be a Banach function space ($BFS$ shortly) on $X$ if the following properties are  satisfied:

(i) $\|f\|_{E} =0$ if and only if $f=0$ $\mu-a.e.$;

(ii) $|g|\leq |f|$  $\mu- a.e.$  implies  that $\|g\|_{X} \leq \|f\|_{X}$;

(iii)  if $0 \leq f_j \uparrow f$ $\mu- a.e.$, the, $\|f_j \|_{E} \uparrow \|f\|_{E}$;

(iv) if $\chi_F \in L^0(\mu)$ is such that $\mu(F)< \infty$, then $\chi_F \in E$;

(v) if $\chi_F \in L^0(\mu)$ is such that $\mu(F)< \infty$, then $\int_F f d\mu \leq C_F \|f\|_{E}$ for all $ f\in E$ and with some  positive constant $C_F$.

For a $BFS$ $E$ it is defined K\"othe dual (or associated)  space $E'$ consists of all $f\in L^0(\mu)$
$$  \|f\|_{E'} = \sup \Big\{  \int\limits_X  f g  d\mu:  \|g\|_{E} \leq 1    \Big\} <\infty.      $$
It is known that the space $E'$ is a Banach function space  (see e.g., \cite{BeSh}, Theorem 2.2).  In Banach function spaces the H\"older inequality holds (see, e.g.,  \cite{BeSh}, Theorem 2.4):
$$  \int\limits_{X} |fg|d\mu \leq  \|f\|_{E}\|g\|_{E'}. $$


For a Banach space $E$ and $0<p<\infty$, the $p$-convexification of $E$ is defined  as follows:

$$  E^p= \{ f: |f|^p \in E\}. $$

$E^p$ can be equipped with the quasi-norm $\|f\|_{E^p}= \| | f|^p \|_{E}^{1/p}$. It can be observed that if $1\leq p<\infty$, then $E^p$ is a Banach space as well. For $1\leq p< \infty$ and  $BFSs$ $E$ and $F$,   we have that $E^{1/p}= F$ if and only if $E= F^p$.

In \cite{Doan1} it was proved the following quantitative variant of the Rubio de Franc\'ia's (\cite{Rubio}) extrapolation theorem
(see also  \cite{KMMIAN} for related topics):

\vskip+0.2cm
\noindent \textbf{ Theorem A (Diagonal Case). }
{\em  Let  $(X,d,\mu)$ be an $SHT$. Suppose that for some family ${\mathcal{F}}$ of pairs of  non-negative measurable functions  $(f,g)$, for some
$p_0\in [1,\infty)$ and all $(f,g) \in {\mathcal{F}}$ and  $w\in A_{p_0}(X)$ the inequality

\begin{equation}\label{Duoandiko-}
    \bigg(\int_X g^{p_0} w\,d\mu\bigg)^{\frac{1}{p_0}}\leq C\,N\big([w]_{A_{p_0}(X)}\big)
            \bigg( \int_X f^{p_0} w\,d\mu\bigg)^{\frac{1}{p_0}}
\end{equation}
holds, where  $N$ is a non-decreasing function and the constant  $C$ does not depend on $(f,g)$ and  $w$.
Then for any  $p$, $1<p<\infty$,  $w\in A_p(X)$ and all $(f,g) \in {\mathcal {F}}$ we have

\begin{equation*}
    \bigg( \int_X g^{p} w \,d\mu   \bigg)^{ \frac{1}{p} }    \leq C \,K(w, \|M\|, p, p_0)
    \bigg( \int_X f^{p} w\,d\mu \bigg)^{\frac{1}{p}},
\end{equation*}
where the positive constant  $C$ is the same as in  \eqref{Duoandiko-}, and

$$  K(w,\|M\|, p, p_0)=\begin{cases}
N \big( [w]_{A_p(X)} \big( 2\|M\|_{ L^p_w(X) \to L^p_w(X) } \big)^{p_0-p} \big), & p<p_0, \\[0.2cm]
N\big( [w]_{A_p(X)}^{ \frac{p_0-1}{p-1} } \big( 2 \|M\|_{ L^{p'}_{w^{1-p'}(X)}\to L^{p'}_{ w^{1-p'} }( X)} \big)^{
\frac{p-p_0}{p-1} }\big), & p>p_0.
\end{cases}        $$}

\begin{remark}\label{Const-K} By \eqref{1-p'} and \eqref{HuPeRe}  we have that
\begin{equation*}
K(w, \|M\|, p, p_0)\leq K(w, p, p_0),
\end{equation*}
where
$$  K(w, p, p_0) = \begin{cases}
N  \Big( (2\overline{c} p')^{p_0-p} [w]_{A_p(X)}^{(p'-1)(p_0-p)}  \Big),  & p<p_0, \\[0.2cm]
N\Big(  (2\overline{c} p')^{  \frac{p_0-p}{p-1}}  [w]_{A_p(X)}^{  \frac{2p_0 +pp_0+1}{(p-1)^2}} \Big),  & p>p_0.
\end{cases}$$
\end{remark}

\vskip+0.2cm
\noindent \textbf{Theorem B (Off-diagonal case).}
{\em Let  $(X,d,\mu)$ be an $SHT$. Suppose that for pairs of non-negative measurable functions  $(f,g)\in {\mathcal{F}}$,   $p_0\in [1,\infty)$, $q_0\in(0,\infty)$, and all  $w\in \mathcal{A}_{p_0,q_0}(X)$ we have
\begin{equation}\label{Duoandiko-+}
    \bigg(\int_X g^{q_0} w^{q_0}\,d\mu\,\bigg)^{\frac{1}{q_0}}\leq C\,N\big([w]_{{\mathcal{A}}_{p_0,q_0}(X)}\big)
            \bigg(\int_X f^{p_0} w^{p_0}\,d\mu\bigg)^{\frac{1}{p_0}},
\end{equation}
where $N$ is non-decreasing function and the constant  $C$ does not depend on  $(f,g)$ and $w$. Then for all  $p$, $1<p<\infty$,  and  $q$, $0<q<\infty$, such that

$$  \frac{1}{q_0}-\frac{1}{q}=\frac{1}{p_0}-\frac{1}{p},     $$
and all  $w\in \mathcal{A}_{p,q}(X)$ the inequality
\begin{equation*}
    \bigg(\int_X g^{q} w^{q}\,d\mu\bigg)^{\frac{1}{q}}\leq C\,K(w, \| M\|,  p,q, p_0, q_0)\bigg(\int_X f^{p} w^{p}\,d\mu\bigg)^{\frac{1}{p}},
\end{equation*}
is fulfilled where  $C$ is the same constant as in \eqref{Duoandiko-+} and
$$  K(w, \| M\|, p,q, p_0, q_0)\!=\! \begin{cases}\!
        N\Big([w]_{{\mathcal{A}}_{p,q}(X)}\big(2\|M\|_{L^{\gamma q}_{w^q}(X)\to L^{\gamma q}_{w^q}(X)}\big)^{\gamma(q-q_0)}\Big),  q<q_0, \\[0.3cm]
        N\Big([w]_{{\mathcal{A}}_{p,q}(X)}^{\frac{\gamma q_0-1}{\gamma q-1}}
            \big(2\|M\|_{L^{\gamma p'}_{w^{-p'}}(X) \to L^{\gamma p'}_{w^{-p'}}(X)}\big)^{\frac{\gamma(q-q_0)}{\gamma q-1}}\Big),  q>q_0,
\end{cases}        $$
with
\begin{equation}\label{gamma}
\gamma:=\frac{1}{q_0}+ \frac{1}{p'_0}.
\end{equation}  }

\begin{remark}\label{rem:2.3}
From \eqref{HuPeRe} it follows the following estimate:

$$ K(w, \| M \|, p,q, p_0, q_0)\leq K(w,p,q, p_0, q_0), $$
where

\begin{equation*}
    K(w,p,q, p_0, q_0)= \begin{cases}
        \displaystyle N\bigg[\Big(2\overline{c}\Big(1+\frac{q}{p'}\Big)\Big)^{\gamma(q-q_0)}[w^q]_{A_{1+\frac{q}{p'}}(X)}^{1+\frac{\gamma(q-q_0)p'}{q}}\bigg],
                    & q<q_0, \\[0.3cm]
    N\bigg[\Big(2\overline{c}\Big(1+\frac{q}{p'}\Big)\Big)^{\frac{\gamma(q-q_0)}{\gamma q-1}}[w^q]_{A_{1+\frac{q}{p'}}(X)}\bigg], & q>q_0,
        \end{cases}
\end{equation*}
\end{remark}
and  $\overline{c}$ is defined by \eqref{overlinesmall}.

Taking the estimate
$$ [w]_{{\mathcal{A}}_{p,q}} = [w^{q}]_{ A_{1+q/p'} } $$
and Remark \ref{rem:2.3} into account,  Theorem B can be reformulated as follows:
\vskip+0.2cm

{\bf Theorem B'.}  {\em   Let $0<q_0<\infty$. Assume that  for some family ${\mathcal{F}}$ of pairs of non-negative functions $(f,g)$,   for $p_0 \in [1,\infty)$, and for all $w\in  A_{ 1+ q_0/(p_0)' }$ the inequality
\begin{equation}\label{Duoandiko-+*}
\bigg(\int\limits_{X} \!\! g^{q_0}(x) w (x) \;d\mu \bigg)^{\frac{1}{q_0}} \!\! \leq C N\big([w]_{ A_{ 1+ q_0/(p_0)' }}\big) \bigg(\! \int\limits_{X} \!\! f^{p_0}(x) w^{p_0/q_0}(x) \;d\mu \bigg)^{\frac{1}{p_0}}
\end{equation}
holds, where $N$ is a non-decreasing  function and the constant $C$ does not depend on $(f,g)$ and $w$.  Then for all $1<p<\infty$, $0<q<\infty$ such that
$$ \frac{1}{q_0}- \frac{1}{q} = \frac{1}{p_0}- \frac{1}{p}, $$
for all $w\in {A}_{ 1+ q/p' }$ and all $(f,g)\in {\mathcal{F}}(X\times Y)$ we have
\begin{equation}\label{Duoandiko--}
\bigg(\int\limits_{X} g^{q}(x) w(x) \; d\mu \bigg)^{\frac{1}{q}} \leq C  K(w,p,q, p_0, q_0) \bigg( \int\limits_X f^p(x) w^{p/q}(x) \;d\mu \bigg)^{\frac{1}{p}},
\end{equation}
where $C$ is the same constant as in (\ref{Duoandiko-+*}),
 \begin{equation}\label{estimate}
K(w,p,q, p_0, q_0)= \begin{cases}
        \displaystyle N \bigg[ \Big( 2\overline{c}  \Big(1+\frac{q}{p'} \Big)  \Big)^{\gamma(q-q_0)} [w]_{ A_{1+\frac{q}{p'}} (X)}
        ^{1+\frac{\gamma p'(q_0-q)}{q}} \bigg],  q<q_0, \\[0.3cm]
    N\bigg[ \Big( 2\overline{c} \Big( 1+\frac{q}{p'} \Big) \Big)^{ \frac{\gamma(q-q_0)}{\gamma q-1} }
    [w]_{A_{ 1+\frac{q}{p'} }} \bigg],  q>q_0,
\end{cases}
\end{equation}
with $\overline{c}$  and $\gamma$  defined by \eqref{overlinesmall} and \eqref{gamma}, respectively. }
\vskip+0.2cm

Finally we mention that in the sequel under the symbol $f(t)\approx g(t)$ we mean that there is a positive constants $c$  independent of $t$ such that
$\frac{1}{c}f(t) \leq g(t) \leq  c f(t) $.

\section{Extrapolation in Banach Function Spaces}

One of our aims in this paper is to establish  weighted extrapolation in Banach function spaces ($BFS$ shortly) defined on an $SHT$. This will enable us to get quantitative
estimates in the case of weighted  Lorentz spaces $L^{p, s}_w(X)$ which will be applied to get appropriate results in grand Lorentz spaces  and consequently,  the boundedness
of operators of Harmonic Analysis in these spaces.

We say that a $BFS$ denoted by $E$ belongs to ${\mathbb{M}}(X)$ if the maximal operator $M$ is bounded in $E$.
\vskip+0.1cm

For extrapolation results on $BFS$s we refer to \cite{CFMP}, \cite{CGMP}, \cite{Ho} (see also \cite{CMP} for related topics).  It should be emphasized that  in \cite{Ho}  the author studied weighted extrapolation problem in mixed norm spaces.

Before formulating the main results recall that according to  Remark \ref{mainremark} we denote by  $C$ constants depending on
  $p,s, \cdots, $  and having property \eqref{const}.

\begin{theorem}\label{Theorem 1}[Diagonal Case] Let ${\mathcal{F}}$ be a family of pairs $(f,g)$ of measurable non-negative functions $f,g$
defined on  $X$.  Suppose that there is a positive constant $C$ such that for  some  $1< p_0 <\infty$,  for every $w\in A_{p_0}(X)$ and all
$(f, g)\in {\mathcal{F}}$, the one-weight inequality holds
\begin{equation}\label{Rub}
\bigg( \!\! \int\limits_{X} \!\! g^{p_0}(x) w(x) \;d\mu(x)\bigg)^{\frac{1}{p_0}} \!\! \leq C  N \big([w]_{ A_{p_0} } \big)
\bigg( \!\! \int\limits_{X}\!\!  f^{p_0}(x) w(x) \;d\mu (x)\bigg)^{\frac{1}{p_0}},
\end{equation}
where  $N(\cdot)$ is a non-negative and  non-decreasing function. Suppose that $E$ is a $BFS$ and that there exists
$1<  q_0< \infty$ such that $E^{1/q_0}$ is again  a $BFS$. If $(E^{1/q_0})' \in {\mathbb{M}}(X)$, then for any $(f,g) \in {\mathcal{F}}$ with
$\| g\|_{E}< \infty$,
$$  \| g\|_{E} \leq 4 C K(\| M\|_{(E^{1/q_0})'}, p, p_0) \| f\|_{E}, $$
where  $K$ is defined as follows:

\begin{equation}\label{kle}
K \big( \| M\|_{ (E^{1/q_0})' }, q_0, p_0 \big) =
\begin{cases}
N  \Big( (2\overline{c} (q_0)')^{p_0-p} \| M \|_{  (E^{1/q_0})'  }^{ ((q_0)'-1)(p_0-(q_0)') }  \Big),  \;\;  q_0 <p_0,
\\
N \Big(  (2\overline{c} (q_0)')^{    \frac{p_0-q_0}{q_0p-1}  }  \|  M \|_{ (E^{1/q_0})' }^{ \frac{2p_0 +q_0 p_0+1}{(q_0-1)^2}  }\Big),
\;\;  q_0>p_0
\end{cases}
\end{equation}
and $C$ is the same as in \eqref{Rub}.
\end{theorem}

\begin{theorem}\label{Theorem 2}[Off-diagonal Case] Let ${\mathcal{F}}$ be a family of pairs $(f,g)$ of measurable non-negative functions $f,g$ on $X$.
Suppose that for  some  $1\leq  p_0,  q_0<\infty$ and for every $w\in A_{1+ q_0/(p_0)'}(X)$ and $(f, g)\in {\mathcal{F}}$,
the one-weight inequality holds

\begin{equation}\label{Rub1}
\bigg( \!\! \int\limits_{X} \!\! g^{q_0}(x) w(x) \;d\mu(x)\bigg)^{ \frac{1}{q_0} } \!\!
\leq  C
N \big([w]_{A_{1+q_0/(p_0)'}(X)}\big)  \bigg( \!\! \int\limits_{X}\!\!  f^{p_0}(x) w^{p_0/q_0}(x) \;d\mu(x)\bigg)^{\frac{1}{p_0}}
\end{equation}
with a positive constant $C$ independent of $(f,g)$ and $w$, and with some positive non-decreasing function $ N(\cdot)$. Suppose that $E$ and $\overline{E}$ are $BFS$s such that there exist
$1 < \widetilde{p}_0 < \infty$,
$1 < \widetilde{q}_0< \infty$ satisfying  the conditions
\begin{equation}\label{p0}
\frac{1}{\widetilde{p}_0}- \frac{1}{\widetilde{q}_0}= \frac{1}{p_0}- \frac{1}{q_0},
\end{equation}

\begin{equation}\label{EX}
\overline{E}(X)^{1/\widetilde{q}_0 }, \;\; E(X)^{ 1/\widetilde{p_0}} \; \text{are  BFSs}
\end{equation}
and

\begin{equation}\label{EXX}
\big( \overline{E}(X)^{ 1/\widetilde{q}_0 } \big)'=  \Big[  \big( E(X)^{ 1/\widetilde{p}_0 } \big)'\Big]^{ \widetilde{p}_0/ \widetilde{q}_0 }.
\end{equation}

If $\Big( \overline{E}^{ 1/ \overline{q}_0 } \Big)' \in {\mathbb{M}}(X)$, then for any $(f,g) \in {\mathcal{F}}$ with $\| g\|_{\overline{E}}
< \infty$,  we have

$$  \| g\|_{\overline{E}} \leq 4 C  \Big(  \overline{K}( \|M\|, \widetilde{p}_0,\widetilde{q}_0, p_0, q_0 )\Big)^{\widetilde{q}_0}
\| f\|_{E}, $$
where the constant $C$ is the same as in \eqref{Rub1},


\begin{align*}
&\overline{K}(\|M\|, \widetilde{p}_0,\widetilde{q}_0, p_0, q_0)  \\ & =
C \begin{cases}
        \displaystyle N \bigg[ \Big( 2\overline{c}  \Big(1+\frac{\widetilde{q}_0}{(\widetilde{p}_0)'} \Big)
        \Big)^{\gamma(\widetilde{q}_0-q_0)} \| M\|_{(E^{1/q_0})'}
        ^{1+\frac{\gamma (q_0-(\widetilde{p}_0)')}{ \widetilde{q}_0}} \bigg],  \widetilde{q}_0<q_0, \\[0.3cm]
    N\bigg[ \Big( 2\overline{c} \Big( 1+\frac{\widetilde{q}_0}{(\widetilde{p}_0)'} \Big) \Big)^{
    \frac{\gamma(\widetilde{q}_0-q_0)}{\gamma \widetilde{q}_0-1} }
    \| M\|_{ (E^{1/q_0})' } \bigg],  \widetilde{q}_0>q_0,
    \end{cases}
\end{align*}
with $\gamma$ defined by \eqref{gamma}.
\end{theorem}

{\em Proof of Theorem} \ref{Theorem 1}  We use  the arguments from the proof of Theorem 3.2 in \cite{Ho}.
Take $q_0$ so that the conditions of the theorem are  satisfied.   By using Theorem A together with Remark \ref{Const-K} we have that for any $w\in A_1(X)$,
\begin{equation*}\label{Ineq}
\bigg( \! \int\limits_{X} \!\!g^{q_0}(x) w(x) d\mu(x)\bigg)^{\frac{1}{q_0}}   \!\! \leq C K
\big(w, q_0, p_0\big) \bigg( \int\limits_{X} \!\! f^{q_0}(x) w(x) \;d\mu(x) \! \bigg)^{\frac{1}{q_0}},
\end{equation*}
where the constant $C$ is the same as in \eqref{Rub} and

$$  K \big(w, q_0, p_0\big)= \begin{cases}
N  \Big( (2\overline{c} (q_0)')^{p_0-q_0} [w]_{A_1(X)}^{((q_0)'-1)(p_0-q_0)}  \Big),  & q_0<p_0, \\[0.2cm]
N\Big(  (2\overline{c} (q_0)')^{  \frac{p_0-q_0}{q_0-1}}  [w]_{A_1(X)}^{  \frac{2p_0 +q_0 p_0+1}{(q_0-1)^2}\Big) },  & q_0>p_0
\end{cases}.$$

Let now $F= E^{1/q_0}$.  Then following  to the Rubio de Franc\'ia's  algorithm (\cite{Rubio}),  for any non-negative  measurable functions $h$,   we define
$$ {\mathcal{R}} h (x) = \sum_{k=0}^{\infty} \frac{M^k h(x)}{2^k \|M\|_{F'}^k},\; x\in X,  $$
where $M$ is the  Hardy-Littlewood  maximal  operators defined on $X$;  $M^{k}$ is    $k$-th  iteration of $M$    with $M^{0} h=h$.
It is easy to check that
\begin{equation}\label{h1} h (x) \leq {\mathcal{R}} h (x);  \; \| {\mathcal{R}} h \|_{F'} \leq 2 \| h \|_{F'};\;
[{\mathcal{R}}h]_{A_1(X)} \leq \| M \|_{F'}.
\end{equation}

Further, from the definition of the K\"othe dual space, there exists  a non-negative $\mu-$ measurable
function $h \in F'(X)$  with $\| h \|_{F'(X)} \leq 1$  such that
\begin{equation*}\label{**}
 \| g \|_{E}^{q_0} = \| g^{q_0} \|_{F} \leq 2 \int\limits_{X}|g(x)|^{q_0} h(x)d\mu(x).
\end{equation*}
Further, by the first inequality of \eqref{h1} we have that
$$ \int\limits_{X}|g(x)|^{q_0} h(x)d\mu(x) \leq \int\limits_{X}|g(x)|^{q_0} ({\mathcal{R}} h)(x)d\mu(x). $$

To apply Theorem A we  show that
$$ \int\limits_{X}|g(x)|^{q_0} ({\mathcal{R}} h)(x)d\mu(x) < \infty. $$

This is true because   the first and second inequalities  of \eqref{h1} with     H\"older's inequality
yield that

\begin{align*}
\int\limits_{X} (g(x))^{q_0}  ({\mathcal{R}} h)(x)  d\mu(x)  \leq    \| g^{q_0} \|_{F} \| Rh  \|_{F'} \leq 2 \|g\|_{E}^{q_0} \|h  \|_{F'}
\leq 2 \| g\|_{E}^{q_0} < \infty.
\end{align*}

Further, by the third inequality of \eqref{h1} we have that ${\mathcal{R}}h \in A_1(X)$. Consequently,


\begin{align*}
& \|g\|^{q_0}_{E}   \leq  2 \int\limits_X g^{q_0} h d\mu \leq 2 \int\limits_X g^{q_0}  ({\mathcal{R}} h ) d\mu
 \leq
2 C K^{q_0} ({\mathcal{R}} h, q_0, p_0) \int\limits_{X} f^{q_0} ({\mathcal{R}} h) d\mu
 \\ & \leq 2 C K^{q_0} ({\mathcal{R}} h, q_0, p_0) \| f^{q_0} \|_{F} \| {\mathcal{R}} h \|_{F'} \leq 4 C K^{q_0} ( {\mathcal{R}} h, q_0, p_0)
\| f \|^{q_0}_{E} \| h \|_{F'}
 \\  &
\leq 4 C K^{q_0} ({\mathcal{R}} h, q_0, p_0) \| f \|^{q_0}_{E},
\end{align*}
where
\begin{equation*}
K ({\mathcal{R}} h, q_0, p_0)=  \begin{cases}
N  \Big( (2\overline{c} (q_0)')^{p_0-p} [{\mathcal{R}} h]_{A_1(X)}^{((q_0)'-1)(p_0-q_0)}  \Big),  & q_0<p_0, \\[0.2cm]
N\Big(  (2\overline{c} (q_0)')^{  \frac{p_0-q_0}{q_0-1}}  [{\mathcal{R}} h]_{A_1(X)}^{  \frac{2p_0 +q_0p_0+1}{(q_0-1)^2} } \Big),  & q_0>p_0.
\end{cases}.
\end{equation*}

Thus, applying the third estimate of \eqref{h1} we find that

$$  \|g\|_{E}   \leq   4C  K \big(  \|M\|_{ (E^{1/q_0})'}, q_0, p_0  \big) \| f \|_{E} $$
with $K (\|M\|_{(E^{1/q_0})'}, q_0, p_0)$ defined by \eqref{kle}.

This complete the proof of the theorem.  $\;\; \Box$

\vskip+0.4cm


{\em Proof of Theorem} \ref{Theorem 2}.  Choose  ${\widetilde{p}}_0$, ${\widetilde{q}}_0$ so that  $p_0 \leq \widetilde{p}_0 < \infty$,
$q_0 \leq \widetilde{q}_0< \infty$,  and conditions  \eqref{p0}, \eqref{EX} and \eqref{EXX} are  satisfied.

Applying  Theorem B'  we have that for any
$w\in A_1$,

\begin{equation*}\label{Duoandiko--}
\bigg(\int\limits_{X} g^{\widetilde{q}_0 }(x) w(x) \; d\mu(x) \bigg)^{\frac{1}{\widetilde{q}_0}} \leq C
 K(w,\widetilde{p}_0,\widetilde{q}_0, p_0, q_0) \bigg( \int\limits_X f^{ \widetilde{p}_0 }(x) w^{\widetilde{p}_0/\widetilde{q}_0}(x)
 \;d\mu(x) \bigg)^{\frac{1}{\widetilde{p}_0}}
\end{equation*}
holds, where $N$ is a non-decreasing  function and the constant $C$ is the same as in \eqref{Rub1}, and

\begin{align*}
 K(w, \widetilde{p}_0, \widetilde{q}_0, p_0, q_0)=  \begin{cases}
\displaystyle N \bigg[ \Big( 2\overline{c}  \Big(1+\frac{\widetilde{q}_0}{(\widetilde{p}_0)'} \Big)
\Big)^{\gamma(\widetilde{q}_0-q_0)} [w]_{ A_{1+\frac{q_0}{(p_0)'}} (X)}^{1+\frac{\gamma (\widetilde{p}_0)'(q_0-\widetilde{q}_0)}{\widetilde{q}_0 }}
\bigg],  \widetilde{q}_0<q_0 \\[0.3cm]
N\bigg[ \Big( 2\overline{c} \Big( 1+\frac{\widetilde{q}_0}{(\widetilde{p}_0)'} \Big) \Big)^{ \frac{\gamma(\widetilde{q}_0-q_0)}{\gamma \widetilde{q}_0-1} }
[w]_{A_{ 1+\frac{q_0}{(p_0)'} }} \bigg],  \widetilde{q}_0 >q_0.
\end{cases}
\end{align*}

Let now  $\overline{F}=  \overline{E}^{ 1/ \widetilde{q}_0  } $ and  $F=  E^{ 1/\widetilde{p}_0 }$.  Then following again  to the Rubio de
Franc\'ia's algorithm,  for any non-negative  measurable function $h$,   we introduce
$$ {\mathcal{R}} h (x) = \sum_{k=0}^{\infty} \frac{M^k h(x)}{2^k \|M\|_{\overline{F}'}^k},\;  x\in X, $$
where, as before,  $M$ is the  Hardy-Littlewood  maximal  operators defined on $X$. Further, it can be checked that

\begin{equation}\label{h1+}  h (x) \leq {\mathcal{R}} h (x); \; \| {\mathcal{R}} h \|_{ \overline{F}' } \leq 2 \| h \|_{\overline{F}'},\;
[{\mathcal{R}}h]_{A_1(X)} \leq \| M \|_{\overline{F}'}.
\end{equation}

Let us take now   non-negative $\mu-$ measurable  function $h \in \overline{F}'(X)$ with $\| h \|_{\overline{F}'(X)} \leq 1$  such that
$$
\| g \|_{ \overline{E}}^{\widetilde{q}_0}   =
\| g \|_{  \overline{F}^{ \widetilde{q}_0} }^{\widetilde{q}_0}=
\| g^{\widetilde{q}_0}\|_{ \overline{F}} \leq 2 \int\limits_{X} g^{{\widetilde{q}}_0}(y)  h(y) d\mu(y) \leq 2 \int\limits_{X} g^{{\widetilde{q}}_0}(y)  ({\mathcal{R}} h)(y) d\mu(y).
$$
The latter estimate follows from the first inequality in \eqref{h1+}.  Further, observe that   H\"older's inequality and the second estimate of \eqref{h1+} yield that

\begin{align*}
\int\limits_{X} (g(x))^{\widetilde{q}_0} ({\mathcal{R}} h)(x)   d\mu(x)
\leq 2 \| g^{\widetilde{q}_0} \|_{\overline{F}}\|{\mathcal{R}} h \|_{\overline{F}'} \leq  4 \| g^{\widetilde{q}_0}\|_{\overline{F}} =
4 \| g\|^{\widetilde{q}_0}_{\overline{E}}  < \infty.
\end{align*}

By  using the fact that  ${\mathcal{R}} h \in A_1(X)$,  H\"older's inequality,   Theorem B' and  the third inequality of \eqref{h1+} we find  that

\begin{align*}
\|g\|^{\widetilde{q}_0}_{E} & \leq  2 \int\limits_{X} (g(x))^{\widetilde{q}_0} {\mathcal{R}}h(x) d\mu(x)  \\
&\leq  2 C  \Big( K( {\mathcal{R}}h, \widetilde{p}_0, \widetilde{q}_0, p_0, q_0)\Big)^{\widetilde{q_0}}
\bigg(  \int\limits_{X} (f(x))^{\widetilde{p}_0} \big({\mathcal{R}}h(x)\big)^{\widetilde{p}_0/ \widetilde{q}_0}  d\mu (x) \bigg)^{  \widetilde{q}_0/ \widetilde{p}_0  }\\
&\leq  2 C   \Big(K({\mathcal{R}}h,, \widetilde{p}_0, \widetilde{q}_0, p_0, q_0)\Big)^{\widetilde{q_0}}
\| f^{\widetilde{p}_0} \|^{  \widetilde{q}_0/ \widetilde{p}_0}_{F}
\| ({\mathcal{R}} h)^{ \widetilde{p}_0/ \widetilde{q}_0}  \|_{F'}^{  \widetilde{q}_0/\widetilde{p}_0 } \\
&= 2  C  \Big(K({\mathcal{R}}h,, \widetilde{p}_0, \widetilde{q}_0, p_0, q_0)\Big)^{\widetilde{q_0}}   \| f\|_{E}^{ \widetilde{q}_0 }
 \| {\mathcal{R}} h \|_{ \overline{F}'} \\
&\leq  4 C  \Big(K({\mathcal{R}}h,, \widetilde{p}_0, \widetilde{q}_0, p_0, q_0)\Big)^{\widetilde{q_0}}   \| f\|_{E}^{ \widetilde{q}_0 }
 \|h \|_{ \overline{F}'} \\
& \leq   4  C  \Big(K({\mathcal{R}}h, \widetilde{p}_0, \widetilde{q}_0, p_0, q_0)\Big)^{\widetilde{q_0}}
\| f\|_{E}^{\widetilde{q}_0}.
\end{align*}
where   $K( {\mathcal{R}}h, \widetilde{p}_0, \widetilde{q}_0, p_0, q_0)$ is given by

\begin{align*}\label{over}
K\big({\mathcal{R}}h, \widetilde{p}_0, \widetilde{q}_0, p_0, q_0\big)  =
\begin{cases}
\displaystyle N \bigg[ \Big( 2\overline{c}  \Big(1+\frac{\widetilde{q}_0}{(\widetilde{p}_0)'} \Big)
\Big)^{\gamma(\widetilde{q}_0-q_0)}  \|  {\mathcal{R}} h\|_{A_1}^{1+\frac{\gamma \widetilde{q}_0(q_0-\widetilde{q}_0)}{(\widetilde{p}_0)'}} \bigg],  \widetilde{q}_0<q_0, \\[0.3cm]
N\bigg[ \Big( 2\overline{c} \Big( 1+\frac{\widetilde{q}_0}{(\widetilde{p}_0)'} \Big) \Big)^{ \frac{\gamma(\widetilde{q}_0-q_0)}{\gamma \widetilde{q}_0-1} }
 \|  {\mathcal{R}} h\|_{A_1} \bigg],  \widetilde{q}_0 >q_0,
 \end{cases}
\end{align*}

Further,  by virtue of the third inequality of \eqref{h1+}  we have that

\begin{equation*}\label{over}
K({\mathcal{R}}h, \widetilde{p}_0, \widetilde{q}_0, p_0, q_0)  \leq \widetilde{K} (\|M\|, \widetilde{p}_0, \widetilde{q}_0, p_0, q_0).
\end{equation*}

Finally we get the desired result .  $\;\; \Box$

\section{Weighted Extrapolation in Lorentz Spaces}\label{Lorentz}

In this section we prove weighted extrapolation results for weighted Lorentz spaces. Initially   let us recall the following result regarding
the boundedness of $M$ in weighted Lorentz spaces (see \cite{ChHuKu} for ${\mathbb{R}}^n$ and \cite{GeKo} for an $SHT$):

{\bf Theorem C.} Let $1< p,s< \infty$. Then $M$ is bounded in  $L^{p,s}_w(X)$ if and only if $w\in A_p(X)$.

We need to calculate the quantitative upper bound  of the norm of maximal operator in weighted Lorentz spaces.

\begin{proposition}
Let $1< p,s< \infty$ and let $w\in A_p(X)$. Then the following estimate holds:

$$  \|M f \|_{ L^{p,s}_w } \leq C 2^{1/p} (\varepsilon_0)^{-1} \Big[  p [w]_{A_{p- \varepsilon_0}} + (p-\varepsilon)
[w]_{A_{p+\varepsilon_0}} \Big], $$
where $C$ is a structural constant and
\begin{equation}\label{varepsilon0}
\varepsilon_0= \frac{p-1}{1+ \tau_{\kappa, \mu}[w]_{A_p}},
\end{equation}
with $\tau_{\kappa, \mu}$ defined in (8).
\end{proposition}
\begin{proof}
Let $w\in A_p$. Then  $w\in A_{p-\varepsilon_0}$ with $ \varepsilon_0$ defined by \eqref{varepsilon0} (see, e.g. \cite{HuPeRe}).
By monotonicity property of $A_p$ classes we have that $w\in A_{p+\varepsilon_0}$. Hence by \eqref{HuPeRe}:

$$ \| M\|_{L^{p- \varepsilon_0}_w(X) \mapsto L^{p- \varepsilon_0, \infty}_w(X)}\leq \overline{c} (p- \varepsilon_0)'
[w]_{A_{p- \varepsilon_0} (X)} $$
and
$$ \| M\|_{L^{p+ \varepsilon_0}_w(X) \mapsto L^{p+ \varepsilon_0, \infty}_w(X)}\leq C
[w]_{A_{p+ \varepsilon_0} (X)}, $$
where $\overline{c}$ is defined by \eqref{overlinesmall}.

Consequently, by virtue of the Marcinkiewicz interpolation theorem in Lorentz spaces (see \cite{SS}, Ch. V) we find that
\begin{align*}
&\| M\|_{L^{p,s}_w(X) \mapsto L^{p,s}_w(X)} \\ & \leq C 2^{1/p} \varepsilon_0^{-1}\bigg[ p \| M\|_{L^{p- \varepsilon_0}_w(X)
\mapsto L^{p- \varepsilon_0, \infty}_w(X)}
 + (p- \varepsilon_0)
 \| M\|_{ L^{p+ \varepsilon_0}_w(X)  \mapsto  L^{p+ \varepsilon_0, \infty}_w(X)} \bigg].
 \end{align*}
This implies that
$$ \| M\|_{L^{p,s}_w(X) \mapsto L^{p,s}_w(X)} \leq C 2^{1/p} \varepsilon_0^{-1}\bigg[ p [w]_{A_{p+\varepsilon_0}}
 + (p- \varepsilon_0)[w]_{A_{p-\varepsilon_0}}
\bigg]. $$
\end{proof}

The next statement will be useful for us:
\begin{proposition}\label{propo0}
Lt $1< p,s< \infty$ and let $w\in A_p(X)$. Then the following estimate holds:
$$  \|w^{-1} M f\|_{L^{p',s'}_w} \leq C 2^{1/p'} (\varepsilon_0)^{-1} \Big[  p' [w]_{A_{p- \varepsilon_0}} + (p-\varepsilon)'
[w]_{A_{p+\varepsilon_0}} \Big] \|w^{-1}  f\|_{L^{p',s'}_w}, $$
where $\varepsilon_0$ is defined by \eqref{varepsilon0}.
\end{proposition}
\begin{proof}
Let $w\in A_p$. Then,  $w\in A_{p-\varepsilon_0}$,  $w\in A_{p+\varepsilon_0}$, where  $ \varepsilon_0$ is defined by \eqref{varepsilon0}. Hence,
$$ w^{1- (p-\varepsilon_0)'}\in A_{(p-\varepsilon_0)'};\;\; w^{1- (p+\varepsilon_0)'}\in A_{(p+\varepsilon_0)'}. $$
Consequently, by \eqref{HuPeRe},
$$ \| M \|_{   L^{(p- \varepsilon_0)' }_{  w^{1- (p- \varepsilon_0)'}} (X) \to L^{ (p-\varepsilon_0)' }_{  w^{1- (p- \varepsilon_0)'}} (X)   }
 \leq
\overline{c} (p-\varepsilon_0) \Big[ w^{1- (p-\varepsilon_0)'} \Big]_{ A_{(p-\varepsilon_0)'}(X) }  ^{1 / ((p-\varepsilon_0)'-1)} $$
and
$$ \| M \|_{   L^{(p+ \varepsilon_0)' }_{  w^{1- (p+ \varepsilon_0)'}} (X) \to L^{ (p+\varepsilon_0)' }_{  w^{1+ (p- \varepsilon_0)'}} (X)   }
 \leq
\overline{c} (p+\varepsilon_0) \Big[ w^{1- (p+\varepsilon_0)'} \Big]_{ A_{(p+\varepsilon_0)'}(X) }  ^{1 / ((p+\varepsilon_0)'-1)} . $$
We can rewrite these estimates as follows:
$$ \| w^{-1}M f \|_{L^{(p- \varepsilon_0)'}_{w}(X)} \leq C_1(w, p, \varepsilon_0)
\| w^{-1} f \|_{L^{(p- \varepsilon_0)'}_{w}(X)} $$
and
$$ \| w^{-1}M f \|_{L^{(p+ \varepsilon_0)'}_{w}(X)} \leq C_2(w, p, \varepsilon_0)
\| w^{-1} f \|_{L^{(p+ \varepsilon_0)'}_{w}(X)}, $$
where
$$  C_1(w, p, \varepsilon_0)=  \overline{c} (p-\varepsilon_0)  \Big[  w^{1- (p-\varepsilon_0)'} \Big]_{  A_{ (p-\varepsilon_0)' }(X)}^{ 1/((p-\varepsilon_0)'-1)}$$
and
$$  C_2(w, p, \varepsilon_0)=  \overline{c} (p+\varepsilon_0)  \Big[  w^{1- (p-\varepsilon_0)'}
\Big]_{  A_{ (p+\varepsilon_0)' }(X)}^{ 1/((p+\varepsilon_0)'-1)}$$
with the constant $\overline{c}$ defined by \eqref{overlinesmall}.


By using Marcinkiewicz interpolation theorem for Lorentz spaces (see \cite{SS}, Ch. V) with respect to sublinear operator
$$  Tf = w^{-1} M f $$
we get

$$ \| w^{-1}M f \|_{L^{p', r}_{w}(X)} \leq C(w, p, \varepsilon) \| w^{-1}M  \|_{L^{p', r}_{w}(X)},  $$
where $1<r< \infty$ and
$$ C(w, p, \varepsilon) = C 2^{1/p'} \varepsilon_0^{-1}\bigg[ p' C_1(w, p, \varepsilon)
 + (p- \varepsilon_0)'  C_2(w, p, \varepsilon) \bigg].   $$

Here we used the fact that

$$  \frac{1}{p'} = \frac{1-t}{(p-\varepsilon)'}+ \frac{t}{(p+\varepsilon)'}, \;\;\; 0< t< 1.  $$

Taking $r= s'$ we get the desired result.
\end{proof}

\begin{theorem}\label{Theorem 3}[Diagonal Case] Let ${\mathcal{F}}$ be a family of pairs $(f,g)$ of measurable non-negative functions $f,g$ defined on
$X$.  Suppose that for  some  $1\leq  p_0 <\infty$,   for every $w\in A_{p_0}(X)$ and all $(f, g)\in {\mathcal{F}}$, the one-weight inequality
\begin{equation}\label{Rub+++}
\bigg( \!\! \int\limits_{X} \!\! g^{p_0}(x) w(x) \;d\mu(x)\bigg)^{\frac{1}{p_0}} \!\! \leq C  N \big( [w]_{ A_{p_0} } \big)
\bigg( \!\! \int\limits_{X}\!\!  f^{p_0}(x) w(x) \;d\mu (x)\bigg)^{\frac{1}{p_0}}
\end{equation}
holds with a positive non-decreasing function $N(\cdot)$ and some positive constant $C$ which does not depend on $(f,g)$ and $w$.
Then for any $1<p,s<\infty$,  for all $(f,g) \in {\mathcal{F}}$ and any $w\in A_p(X)$,
$$  \| g\|_{L^{p,s}} \leq 4 C K_1 \Big(   \| M\|_{ L^{ (p/q_0)'}},    p,s  \Big) \| f\|_{L^{p,s}}, $$
where the constant $C$ is the same as in \eqref{Rub+++} and
\begin{align*}
K_1(\| M\|_{  {\widetilde{L}}^{  (p/q_0)',  (s/q_0)'  }  _w},  p,s) =   \!
\begin{cases}\!
N  \Big( (2\overline{c} (q_0)')^{p_0-p}  \|M\|_{ {\widetilde{L}}^{  (p/q_0)',  (s/q_0)'  }_w} ^{((q_0)'-1)(p_0-(q_0)')}  \Big),
 q_0 <p_0, \\[0.2cm]
N\Big(  (2\overline{c} (q_0)')^{  \frac{p_0-q_0}{q_0p-1}}
 \|M\|_{ \widetilde{L}^{(p/q_0)',  (s/q_0)'}_w}^{  \frac{2p_0 +q_0 p_0+1}{(q_0-1)^2} }\Big),   q_0>p_0.
\end{cases}
\end{align*}
with non-decreasing $N$ and $q_0\in (1,p)$, and

\begin{equation*}\label{tilda}
{\widetilde{L}}^{  (p/q_0)',  (s/q_0)'  }_w   =
\bigg\{  f: X \mapsto {\mathbb{R}}:  \Big\| \frac{1}{w} f \Big\|_{ L^{(p/q_0)',  (s/q_0)' }_w} < \infty \bigg\}.
\end{equation*}
\end{theorem}

\begin{proof}
Let $1<p< \infty$ and let $w\in A_p(X)$.  Then $w\in A_{p-\varepsilon_0}(X)$ with $\varepsilon_0$ equal to the expression given
by \eqref{varepsilon0}. Take $q_0$ so that $p-\varepsilon_0 < p/q_0$. Then by monotonicity property of Muckenhoupt classes, $w\in A_{p/q_0}$. Due to
Proposition \ref{propo0} we find that
\begin{align*}
&\|w^{-1} M  f\|_{L^{(p/q_0)', (s/q_0)'}_w}  \\&
\leq C 2^{1/(p/q_0)'} (\varepsilon_0)^{-1} \Big[  \big(\frac{p}{q_0}\big)'
[w]_{A_{\frac{p}{q_0}- \varepsilon_0}} +
\big(\frac{p}{q_0} -\varepsilon\big)'
[w]_{A_{\frac{p}{q_0} +\varepsilon_0}} \Big] \|w^{-1}  f\|_{L^{(p/q_0)', (s/q_0)'}_w} .
\end{align*}
Now the result follows from Theorem \ref{Theorem 1}.
\end{proof}

\begin{remark} \label{main1}
Taking  the proof of Theorem 4.1 into account we find that
$$ \|  M  \|_{{\widetilde{L}}^{  (p/q_0)',  (s/q_0)'  }_w} \leq C(p, q_0, \varepsilon_0,
[w]_{A_{\frac{p}{q_0} +\varepsilon_0}}, [w]_{A_{\frac{p}{q_0} +\varepsilon_0}}),  $$
where
$$ \sup_{ 0< \varepsilon< \delta_0 }  C \big(  p-\varepsilon, q_0,  \varepsilon_0,   [w]_{  A_{ \frac{p-\varepsilon}{q_0} +\varepsilon_0}  },
[w]_{ A_{ \frac{p-\varepsilon}{q_0} +\varepsilon_0} } \big)< \infty $$
for some small positive number $\delta_0$.
\end{remark}

\begin{theorem}\label{Theorem 4}[Off-diagonal Case] Let ${\mathcal{F}}$ be a family of pairs $(f,g)$
of measurable non-negative functions $f,g$ on $X$.  Suppose that for  some  $1<  p_0, q_0<\infty$ and for every $w\in A_{1+ q_0/ (p_0)'}(X)$ and
$(f, g)\in {\mathcal{F}}$, the one-weight inequality

\begin{equation*}
\bigg( \! \int\limits_{X} \!\! g^{q_0}(x) w(x) \;d\mu(x)\bigg)^{ \frac{1}{q_0} } \!\! \leq C
N \big([w]_{A_{1+ q_0/ (p_0)'}(X)}\big)  \bigg(\! \int\limits_{X} \!  f^{p_0}(x) w^{p_0/q_0}(x) \;d\mu(x)\bigg)^{\frac{1}{p_0}}
\end{equation*}
holds with a positive constant $C$ independent of $(f,g)$ and $w$,  and some non-decreasing positive function $N(\cdot)$. Suppose that $1<p,q,r,s< \infty$ are
chosen so that

\begin{equation*}\label{difference}
\frac{1}{p}- \frac{1}{q}= \frac{1}{s}- \frac{1}{r}= \frac{1}{p_0}- \frac{1}{q_0}.
\end{equation*}
Then for all $w\in A_{1+p/q'}$ and all $(f,g) \in {\mathcal{F}}$ we get

$$ \| g\|_{L^{q,r}_w(X)} \leq\overline{K}(\|M\|, p, q, r,s) \|w^{\frac{1}{q}- \frac{1}{p}} f\|_{L^{p,s}_{w}(X)}, $$
where

\begin{equation*}\label{over+}
\overline{K}(\|M\|, p, q, r,s ) =C \begin{cases}
        \displaystyle N \bigg[ \Big( 2\overline{c}  \Big(1+\frac{\widetilde{q}_0}{(\widetilde{p}_0)'} \Big)
        \Big)^{\gamma(\widetilde{q}_0-q_0)} \| M\|_{ \big( [L^{q,r}_w]^{ 1/q_0 } \big)'}
        ^{1+\frac{\gamma \widetilde{q}_0(q_0-\widetilde{q}_0)}{(\widetilde{p}_0)'}} \bigg],  \widetilde{q}_0<q_0, \\[0.3cm]
    N\bigg[ \Big( 2\overline{c} \Big( 1+\frac{\widetilde{q}_0}{(\widetilde{p}_0)'} \Big) \Big)^{
    \frac{\gamma(\widetilde{q}_0-q_0)}{\gamma q-1} }
    \| M\|_{\big( [L^{q,r}_w]^{ 1/q_0 } \big)'} \bigg],  \widetilde{q}_0>q_0,
    \end{cases}
\end{equation*}
with $\gamma$ defined by \eqref{gamma} and $\widetilde{q}_0$ is defined so that
\begin{equation}\label{leq}
1< \widetilde{q}_0 < \frac{qp'}{p'+q- \varepsilon_0 p'}
\end{equation}
with $\varepsilon_0$ defined by \eqref{varepsilon0}.
\end{theorem}


\begin{proof}
Let  $1<p,r,q,s< \infty$ are chosen so that  the conditions of the theorem are fulfilled.  Suppose that  $w\in A_{1+p/q'}(X)$.
Then the openness property of Muckenhoupt classes yields  that
$w\in A_{1+p/q'- \varepsilon_0}(X)$, where  $\varepsilon_0$ is defined by \eqref{varepsilon0} but replaces $p$ by $1+ q/p'$.

Choose $\widetilde{p}_0$  and $\widetilde{q}_0$ so that
$$ \frac{1}{p_0}- \frac{1}{q_0}= \frac{1}{\widetilde{p}_0}- \frac{1}{\widetilde{q}_0}  $$
and that \eqref{leq}  holds. In this case, $w\in A_{q/\widetilde{q}_0}$ and $\widetilde{q}_0 < q$. Hence by Proposition \ref{propo0} we find that
$$  \| w^{-1} Mf \|_{L^{(q/q_0)',(r/q_0)'}_w} \leq C (q, r, q_0, [w]_{A_{q/ \widetilde{q}_0 - \varepsilon_0}}, [w]_{A_{q/ \widetilde{q}_0 + \varepsilon_0}})  \| w^{-1} f\|_{L^{(q/q_0)',(r/q_0)'}_w} $$
with
$$   \sup_{0< \varepsilon< \delta_0}  C (q-\varepsilon, r, q_0, [w]_{A_{(q-\varepsilon)/ \widetilde{q}_0 - \varepsilon_0}}, [w]_{A_{(q-\varepsilon)/ \widetilde{q}_0 + \varepsilon_0}})< \infty. $$

Let $\overline{E}= L^{q,r}_w$ and
$$ E= \bigg\{  f: \bigg\|  w^{ \frac{1}{q}- \frac{1}{p} } f \bigg\|_{ L^{r,s}_w } < \infty  \bigg\}. $$


Observe now that
$$ \frac{p_0}{\widetilde{q}_0}\Big( \frac{p}{\widetilde{p}_0}\Big)' =
\Big( \frac{q}{\widetilde{q}_0} \Big)'; \;\; \frac{p_0}{\widetilde{q}_0}\Big( \frac{s}{\widetilde{p}_0}\Big)' =
\Big( \frac{r}{\widetilde{q}_0}\Big)' $$
which, on the other hand, implies that

$$\bigg\| w^{\frac{1}{q}- \frac{1}{p}} f\bigg\|_{ \widetilde{L}^{ p,s, \widetilde{q}_0, \widetilde{p}_0 } _w} =
\bigg\| w^{-1} f\bigg\|_{\overline{L}^{q,r, \widetilde{q}_0} _w},$$

where
$$  \widetilde{L}^{p,s, \widetilde{q}_0, \widetilde{p}_0 }_w = \bigg[ \Big( \Big( L^{p,s}_w\Big)^{1/ \widetilde{p}_0}
\Big)'\bigg]^{\widetilde{p}_0
/ \widetilde{q}_0} ;\;\; \overline{L}^{q,r, \widetilde{q}_0} _w = \bigg[ \Big( L^{q,r}_w \Big)^{ 1/ \widetilde{q}_0 } \bigg]'.  $$

Now the result follows from Theorem  \ref{Theorem 2}.
\end{proof}




\begin{remark} \label{main2}
The proof of Theorem 4.2 yields that
$$ \| M f \|_{ \overline{L}^{q,r, \widetilde{q}_0} _w} \leq
C (q, q_0, \widetilde{q}_0, \varepsilon_0,   w ),  $$
where
$$ \sup_{ 0< \varepsilon< \delta_0 }  C \big(  q - \varepsilon, q_0,  \varepsilon_0, \widetilde{q}_0,   w  \big)< \infty $$
for some small positive number $\delta_0$.
\end{remark}

\section{Extrapolation in Grand Lorentz Spaces}

Applying statements proven in Section \ref{Lorentz} we have the following results regarding grand Lorentz spaces:

\begin{theorem}\label{Theorem 5}[Diagonal Case] Let $w$ be integrable weight on $X$ and let ${\mathcal{F}}$ be a family of pairs $(f,g)$ of measurable non-negative functions $f,g$ defined on
$X$.  Suppose that for  some  $1\leq  p_0 <\infty$,   for every $w\in A_{p_0}(X)$ and all $(f, g)\in {\mathcal{F}}$, the one-weight inequality holds
\begin{equation*}\label{Rub+}
\bigg( \!\! \int\limits_{X} \!\! g^{p_0}(x) w(x) \;d\mu(x)\bigg)^{\frac{1}{p_0}} \!\! \leq C  N \big( [w]_{ A_{p_0} } \big)
\bigg( \!\! \int\limits_{X}\!\!  f^{p_0}(x) w(x) \;d\mu (x)\bigg)^{\frac{1}{p_0}}
\end{equation*}
with some positive constant $C$ which does not depend on $(f,g)$ and $w$, and positive non-decreasing function $N(\cdot)$.
Then for any $1<p<\infty$, $1\leq s < \infty$,  $\theta>0$, $w\in A_p(X)$ and  for all $(f,g) \in {\mathcal{F}}$,
$$  \| g\|_{L^{p),s, \theta}_w} \leq C  \| f\|_{L^{p),s, \theta}}, $$
with the positive constant $C$ independent of $(f,g)$.
\end{theorem}

\begin{proof}
Let $w\in A_p$. By H\"older's inequality and the fact that $w$ is integrable on $X$  it is enough to prove that

$$ \sup_{0< \varepsilon< \sigma_0} \varepsilon^{  \frac{\theta}{p-\varepsilon}  }
 \| g \|_{  L^{p-\varepsilon, s}_w  } \leq C \sup_{0< \varepsilon< \sigma_0} \varepsilon^{  \frac{\theta}{p-\varepsilon}  }
 \|  f \|_{  L^{p-\varepsilon, s}_w  }  $$
for all $(f,g) \in {\mathcal{F}}$ and for some positive constant $\sigma_0$.

Observe that Theorem \ref{Theorem 3} and Remark \ref{main1}
yield that
$$ \varepsilon^{\frac{\theta}{p-\varepsilon}}\| g \|_{  L^{p-\varepsilon, s}_w  } \leq C(w, p, s, \varepsilon) \varepsilon^{  \frac{\theta}{p-\varepsilon}  }
\|  f \|_{  L^{p-\varepsilon, s}_w  }  $$
for all $(f,g) \in {\mathcal{F}}$ and all $w\in A_{p-\varepsilon}$ with $0< \varepsilon< \sigma_0$, where
$$ \sup_{0< \varepsilon< \sigma_0} C(w, p, s, \varepsilon) < \infty. $$
\end{proof}


\begin{theorem}\label{Theorem 6}[Off-diagonal Case] Let $w$ be an integrable weight on $X$.  Let    ${\mathcal{F}}$ be a family of pairs $(f,g)$
of measurable non-negative functions $f,g$ on $X$.  Suppose that for  some  $1< p_0 \leq q_0<\infty$ and for every
$w\in A_{1+ q_0/ (p_0)'}(X)$ and $(f, g)\in {\mathcal{F}}$, the one-weight inequality holds

\begin{equation*}
\bigg( \!\! \int\limits_{X} \!\! g^{q_0}(x) w(x) \;d\mu(x)\bigg)^{ \frac{1}{q_0} } \!\! \leq C
N \big([w]_{A_{1+ q_0/ (p_0)'}(X)}\big)  \bigg( \!\! \int\limits_{X}\!\!  f^{p_0}(x) w^{p_0/q_0}(x) \;d\mu(x)\bigg)^{\frac{1}{p_0}}
\end{equation*}
with a positive constant $C$ independent of $(f,g)$ and $w$,  and some non-decreasing positive function $N(\cdot)$. Suppose that $1<p,q,r,s< \infty$ are
chosen so that
\begin{equation*}\label{difference}
\frac{1}{p}- \frac{1}{q}= \frac{1}{p_0}- \frac{1}{q_0};\;\;  \frac{1}{s}- \frac{1}{r}= \frac{1}{p_0}- \frac{1}{q_0}.
\end{equation*}
Then for all $w\in A_{1+p/q'}$ and all $(f,g) \in {\mathcal{F}}$, we have

$$ \| g\|_{L^{q),r, q \theta/p}_w(X)} \leq C \|f\|_{L^{p),s, \theta}_{w}(X)}, $$
where the positive constant $C$ is independent of $(f,g)$.
\end{theorem}

\begin{proof}
Since $X$ is bounded, by H\"older's inequality  we have that
$$  \| g\|_{L^{q),r, q \theta/p}_w(X)} \leq C \sup_{0< \varepsilon< \varepsilon_0} \varepsilon^{   \frac{ \theta q } {p(q-\varepsilon)}  }
 \| g \|_{  L^{q-\varepsilon, r}_w  }. $$

 Let us set:

 \begin{equation*}\label{Psi(x)}
    \Psi(x):= \Phi(x^{\theta}),\;\;\;\Phi(x):=\Big[\frac{x-q}{1- A(x-q)}+p\Big]^{1-(x-q)A}
\end{equation*}

with a number  $A$ defined by
$$  A:= \frac{1}{p_0}- \frac{1}{q_0}= \frac{1}{p}- \frac{1}{q}=\frac{1}{s}- \frac{1}{r}.      $$

It is easy to check that
$$    \Psi(x)\approx x^{q\theta/p},\;\;\;x\to 0.$$

Hence,  it suffices to show that

$$ \sup_{0< \varepsilon< \varepsilon_0} \Psi(\varepsilon)^{   \frac{1} {q-\varepsilon}  }
 \| g \|_{  L^{q-\varepsilon, r}_w  } \leq C \sup_{0< \eta< \eta_0} \eta^{  \frac{\theta}{p-\eta}  }
 \|  f \|_{  L^{p-\eta, s}_w  }  $$
for all $(f,g) \in {\mathcal{F}}$ and for some positive constant $\varepsilon_0$,
where  $\varepsilon_0\in (0, q-1)$,

Here $\eta$ and $\varepsilon$ satisfy the condition:
$$
\frac{1}{p-\eta}- \frac{1}{q-\varepsilon}=A,
$$
and  $\eta_0$ is chosen so that if $\varepsilon \in (0, \varepsilon_0)$, then $\eta\in (0, \eta_0)$.

Theorem \ref{Theorem 4} and Remark \ref{main2} yield

\begin{align*}
   &\sup\limits_{0<\varepsilon<\varepsilon_0} \Psi(\varepsilon)^{\frac{1}{q-\varepsilon}}\| g\|_{L^{q-\varepsilon, r}_w(X)}   \leq C\sup\limits_{0<\eta<\eta_0} C( w,p-\eta, s, q-\varepsilon, r) \eta^{\frac{\theta}{p-\eta}}\|f\|_{L^{p-\eta, s}_w(X)}  \\ & \leq \overline{C} \|f\|_{\mathcal{L}^{p),s, \theta}_w(X)}.
\end{align*}
Observe that here $0< \varepsilon< \varepsilon_0$ if and only if $0<\eta<\eta_0$. Finally, since  $\Phi(\varepsilon) \approx \varepsilon^{q/p}$ we have the desired result.
\end{proof}

\section{Applications of Extrapolation Results in \\Grand Lorentz Spaces}

Based on extrapolation results we get the  boundedness of integral operators of Harmonic Analysis in grand Lorentz spaces.
In this section  we will assume that $X$ is bounded. We  denote by ${\mathcal{D}}(X)$ the class of bounded functions on $X$.

\subsection{Maximal, fractional and singular integral operators}

Let $K$ be the Calder\'on-Zygmund operator defined on an SHT, i.e., $K$ satisfies the following conditions (see, e.g., \cite{Ai}, \cite{CoWe}):

(i) $K$ is linear and bounded in $L^p(X)$ for every $p\in (1, \infty)$;

(ii) there is a measurable function $k: X\times X \mapsto {\mathbb{R}}$ such that for every $f\in D(X)$,
$$  Kf(x) = \int\limits_X k(x,y) f(y) d\mu(y), $$
for a.e. $x \notin \; supp \; f$, where $D(X)$ is the class of bounded functions with compact supports defined on $X$.

(iii) the kernels $k$ and $k^*$ (here $k^*(x,y):= k(y,x)$) satisfy the following pointwise H\"ormander's condition:  there are  positive constants $C$, $\beta$ and $A >1$ such that
$$  |k(x_0,y)- k(x,y)| \leq C \frac{d(x_0, x)^{\beta}}{ \mu(B (x_0, 2d(x_0, y))) d(x_0,y)^{\beta}} $$
holds for every $x_0 \in X$, $r>0$, $x\in B(x_0,r)$, $y\in X\setminus B(x_0,Ay)$;

(iv) there is a positive constant $C$ such that  for all $x,y\in X$,
$$ |k(x,y)|\leq \frac{C}{\mu( B( x, 2 d(x,y))}. $$

The operator $K$ (see, e.g.,  \cite{PrSa}, \cite{EKM} and references therein) is bounded in $L^{p_0}_w(X)$ for $1< p_0< \infty$ and $w\in A_{p_0}(X)$. Moreover,
the following estimate holds:
\begin{equation*}\label{est-ca}
\|Kf\|_{L^{p_0}_w(X)} \leq C_0([w]_{A_{p_0}(X)}) \leq \|f\|_{L^{p_0}_w(X)},\;\; f\in {\mathcal{D}}(X),
\end{equation*}
where $C_0([w]_{A_{p_0}(X)})$ is a constant depending on  $[w]_{A_{p_0}(X)}$ such that the mapping $x\mapsto C_0(x)$ is non-decreasing.



In the next statement by the symbol $I_{\alpha}$ will be denoted the fractional integral operator defined by

\begin{equation*}\label{Ialpha}
I_{\alpha}f(x)=\int_X K_{\alpha}(x,y) f(y)\,d\mu(y), \;\; x\in X,
\end{equation*}
where
$$  K_{\alpha}(x,y)=\begin{cases}
                \mu(B_{xy})^{\alpha-1}, & x\neq y, \\
                \mu\{x\}, & x=y, \;\; \mu\{x\}>0,
                    \end{cases}     $$
$0<\alpha<1$,  $B_{xy}:=B(x,d(x,y))$.

It is known that (see \cite{KoMANA} and   \cite{GeKo}) the following inequality holds:

$$ \| I_{\alpha} (f w^{\alpha})\|_{L^{q,s}_w(X)} \leq C   \| f\|_{L^{p,s}_w(X)}, \;\;\; f\in L^{p,s}_w,  $$
where $1<p< \frac{1}{\alpha}$, $q= \frac{p}{1-\alpha p}$, $1<s<\infty$ and $w\in A_{1+ q/p'}$.

Together with $I_{\alpha}$ we are interested in the related fractional maximal operator
$$   M_{\alpha}f(x) = \sup_{B \ni x} \frac{1}{\mu(B)^{1-\alpha}} \int\limits_B |f(x)| d\mu(y), \;\; 0< \alpha<1. $$

The following pointwise estimate holds for $f\geq 0$
\begin{equation}\label{pointwise++}
M_{\alpha}f(x) \leq C_{\alpha} I_{\alpha}f(x),
\end{equation}
where $C_{\alpha}$ is a positive constant independent of $f$ and $x$.

To prove the statements of this subsection we need the following lemma:

\begin{lemma}\label{Mn}  Let $1< p,s< \infty$ and let $\theta>0$. Then there is a positive constant $C$ such that for all balls $B$ and all $f\in L^p_w(B)$,
\begin{equation*}\label{B}
\| f\|_{L^{p), s, \theta} (B)} \leq   C w(B)^{-1/p} \| f\|_{ L^p_w(B) } \| \chi_B \|_{L^{p),  \theta}(B)}.
\end{equation*}
\end{lemma}

\begin{proof} By using properties (ii) and  (v) of the Lorentz spaces (see Section \ref{Prelim}) with respect to the exponents:
$$  \frac{1}{p-\varepsilon} = \frac{1}{p}+ \frac{\varepsilon}{ p(p-\varepsilon) };  \;\;\;  \frac{1}{s} = \frac{1}{s_1}+ \frac{1}{s_2},  $$
where $\varepsilon \in (0, p-1]$, $p< s_1$, we have

\begin{align*}
& \| f\|_{ L^{p), s, \theta}_w (B)}  = \sup_{0< \varepsilon\leq  p-1} \varepsilon^{\frac{\theta}{p-\varepsilon}} \|f \|_{L^{p-\varepsilon, s}_w}
 \leq C  \sup_{0< \varepsilon\leq  p-1} \varepsilon^{\frac{\theta}{p-\varepsilon}} \|f \|_{L^{p, s_1}_w(B)}
\|\chi_B \|_{ L^{p(p-\varepsilon)/\varepsilon, s_2}_w }
  \\ &  \leq C \|f \|_{ L^{p, s_1}_w }  \sup_{0< \varepsilon< p-1} \varepsilon^{\frac{\theta}{p-\varepsilon}}
w(B)^{ \frac{\varepsilon}{p(p-\varepsilon)} } \leq  C \|f \|_{L^{p}_w} w(B)^{-1/p}    \| \chi_B \|_{L^{p),  \theta}_w}.
\end{align*}

\end{proof}

\begin{theorem}\label{Theorem 7} Let $w$ be an integrable weight on $X$  and let  $1< p,s<\infty$. Suppose that  $\theta>0$. Then $M$ is bounded in $L^{p), s, \theta}_w$ if and only if    $w\in A_p$.
\end{theorem}

\begin{proof}
{\em Sufficiency} follows from Theorem \ref{Theorem 5}. That is why we show only {\em Necessity.}
Suppose that $M$ is bounded in $L^{p), s, \theta}$. Take a ball $B$ and non-negative $f\in L^p_w(B)$. By Lemma \ref{Mn}  we have that
$$\| \chi_B f\|_{L^{p), s, \theta}} \leq   C w(B)^{-1/p} \| \chi_B  f\|_{ L^p_w } \| \chi_B \|_{L^{p) \theta}}   $$
with a positive constant $C$ independent of $B$ and $f$. Since the pointwice inequality
$$  \frac{1}{\mu(B)} \int\limits_B |f(y)| d\mu \leq M(f \chi_B)(x) $$
holds for $x\in B$, then we have that

\begin{align*}
&\| Mf\|_{ L^{p), s, \theta}_w(B)} \geq  \frac{1}{\mu(B)} \int\limits_B \| \chi_B \|_{ L^{p), s, \theta}_w(B) } \\  & \geq
\| \chi_B \|_{ L^{p), s, \theta}_w(B)}
\bigg( \frac{1}{\mu(B)}\int\limits_B |f(y)|d\mu(y) \bigg)   \\  & = \| \chi_B \|_{ L^{p), \theta}_w(B)}
\bigg( \frac{1}{\mu(B)}\int\limits_B |f(y)|d\mu(y) \bigg).
\end{align*}
Consequently, taking the boundedness of $M$ into account we find that

$$ \bigg( \frac{1}{\mu(B)} \int\limits_B  |f(y)| d\mu(y) \bigg) \| \chi_B  \|_{  L^{p), \theta}_w(B) } \leq C  w(B)^{-1/p}\|f \|_{ L^p_w(B)}
\| \chi_B \|_{ L^{p),  \theta}_w }.$$
Now choosing  $f= \chi_B w^{1-p'}$ we conclude that $w\in A_p(X)$.
\end{proof}

\begin{theorem}\label{Theorem 8} Let  $w$ be an integrable weight on $X$ and let  $1< p<\infty$. Suppose that  $\theta>0$ and  $w\in A_p$.
Then there is a positive constant $C$ such that for all  $f\in {\mathcal{D}}(X)$,  the inequality

$$ \| Kf \|_{ L^{p),s, \theta}_w } \leq C \| f \|_{ L^{p),s, \theta}_w } $$
holds. Conversely, if  $H$ be the Hilbert transform on $I:= (0,1)$:

$$ Hf(x)= (p.v) \int\limits_0^1 \frac{f(t)}{x-t} dt,$$
then from the boundedness of $H$ in $L^{p), s, \theta}_w(I)$ it follows that $ w\in A_p(I)$.
\end{theorem}

\begin{proof} The first part (sufficiency) of the statement follows immediately from Theorem \ref{Theorem 5}; that is why we prove the second part of the theorem (necessity).

We follow \cite{KoMeGMJ}. First we show that there is a positive constant $C$ such that for all intervals $J,J' \subset I$,  the following inequality holds:

\begin{equation}\label{****}
\|  \chi_J  \|_{   L^{  p), \theta   }_w(I)   } \leq C \|  \chi_{J'}  \|_{   L^{ p), \theta }_w(I)  },
\end{equation}
where $J:= (a,b)$ with $b-a\leq 1/4$, and

$$ J'=
\begin{cases}  (b, 2b-a) \;  \text{if}\;   (b, 2b-a)  \subset I,  \\
(2a-b, a) \; \text{if}\;  (2a-b, a) \subset I \; \text{and} \; (b, 2b-a) \cap   I^c  \neq \emptyset.
\end{cases}
$$

Indeed, without loss of generality suppose that $J'= (b, 2b-a)$. Then for $f= \chi_{J'}$ and $x\in J$,

$$ \| Hf\|_{L^{p), \theta}_w(J)} \geq \frac{1}{2} \| \chi_J \|_{ L^{p), \theta}_w(I) }. $$
On the other hand, observe that
$$ \| f\|_{L^{p), \theta}_w(J)} = \| \chi_{J'} \|_{L^{p), \theta}_w(I)}. $$
Consequently, due to the boundedness of $H$ we have \eqref{****}.

Arguing now as in the proof of Theorem \ref{Theorem 7} for intervals $J$ and $J'$ and by using Lemma \ref{Mn} we get the condition $w\in A_p(I)$.
\end{proof}

\begin{theorem}\label{Theorem 9} Suppose that  $0<\alpha< 1$ and let  $0< p,s< 1/\alpha$. Let $w$ be an integrable weight on $X$, and let $\theta>0$. We set $q= \frac{p}{1-\alpha p}$, $r= \frac{s}{1-\alpha s}$. Then the following statements are equivalent:

\rm{(i)}

There is a positive constant $C$ such that for all $f\in L^{p), s, \theta}_w$,
$$ \|I_{\alpha}(w^{\alpha} f) \|_{L^{q), r, q \theta/p}_w} \leq C  \| f \|_{L^{p), s,  \theta}_w}; $$

\rm{(ii)}

There is a positive constant $C$ such that for all $f\in L^{p), s, \theta}_w$,
$$ \|M_{\alpha}(w^{\alpha} f) \|_{L^{q), r, q \theta/p}_w} \leq C  \| f \|_{L^{p), s,  \theta}_w}; $$

\rm{(iii)}
$w\in A_{1+ p/q'}$.
\end{theorem}

\begin{proof}
First we will show that (iii) $\Rightarrow$ (i).  Let $w\in A_{1+ p/q'}$. Since
$$\frac{1}{p}- \frac{1}{q}= \frac{1}{s}- \frac{1}{r} = \alpha,$$
due to  Theorem \ref{Theorem 6} we get
$$ \| I_{\alpha} f \|_{L^{q),r, q \theta/p}_w(X)} \leq C \| w^{\frac{1}{q}-\frac{1}{p}} f\|_{L^{p),s, \theta}_{w}(X)}=
C\| w^{-\alpha} f\|_{L^{p),s, \theta}_{w}(X)} $$
provided that the right-hand side norm is finite.

The latter inequality is equivalent to
$$ \| I_{\alpha} (w^{\alpha} f) \|_{L^{q),r, q \theta/p}_w(X)} \leq C \| f\|_{L^{p),s, \theta}_{w}(X)}.  $$


Since (i) $\Rightarrow$ (ii)  by the pointwise inequality \eqref{pointwise++}, it suffices to show that (ii)  $\Rightarrow$ (iii). We follow the arguments of the proof of Theorem 3.1 from \cite{Mes}.

Observe that (ii) is equivalent to the inequality

\begin{equation}\label{**}
\|M_{\alpha}(fw^{\alpha})\|_{ L_{w}^{q),r, \psi(x)}(X)}\leq  c \|f\|_{L_{w}^{p),s, \theta}(X)},
\end{equation}
where

$$\psi(t):=\varphi(t^{\theta}), \;\;\;\varphi(t):=\left[\frac{t-q}{1-\alpha(t-q)}+p\right]^{1-(t-q)\alpha}.
$$

This follows from the fact that $\varphi(t)\approx t^{q/p}$ as $t\to 0$.

 Let (ii) (i.e., equivalently  \eqref{**})  holds.  Let us take a ball $B \subset X$ and
$f=\chi_{B}w^{-\alpha-p'/q}$. Then for $x\in B$, we get that
$$
M_{\alpha}(w^{\alpha}f)(x)\geq\frac{1}{\mu(B)^{1-\alpha}}\int\limits_{B}w^{\alpha}f d\mu =\frac{1}{\mu(B)^{1-\alpha}}
\int\limits_{B}w^{-p'/q} d\mu.
$$
Hence,

$$ \|M_{\alpha}(w^{\alpha}f)\|_{{L_{w}^{q),r, \psi(x)}(X)}}\geq
\mu(B)^{\alpha-1}\left(\int\limits_{B} w^{-p'/q}d\mu \right)\|\chi_{B}\|
_{L_{w}^{q),s, \psi(x)}(X)}.
$$
Further, by Lemma \ref{Mn}  we find that
\begin{align*}
& \mu(B)^{\alpha-1} \left( \int\limits_{X} w^{-p'/q}d\mu \right) \| \chi_{B} \|_{ L_{ w}^{q),r, \psi(x) }(B) }
\\ &
= \mu(B)^{\alpha-1} \left( \int\limits_{X} w^{-p'/q}d\mu\right) \| \chi_{B} \|_{ L_{ w}^{q),  \psi(x) }(B) }
\\ &
\leq c \|f\|_{L^{p),s, \theta}_w(X)} \leq c(w(B))^{-\frac{1}{p}}
\left(\int\limits_{B}|f(y)|^{p}w(y)d\mu(y) \right)^{\frac{1}{p}}
\|\chi_{B}\| _{L_{w}^{p),  \theta}(X)}
\\ &
= c w(B)^{-\frac{1}{p}} \left(\int\limits_{B}w^{-p'/q}\right)^{1/p}
\|\chi_{B}\| _{L_{w}^{p),  \theta}(X)}.
\end{align*}

It is easy to see that there is a number $\eta_{J}$
depending on $J$ such that $0<\eta_{J}\leq p-1$ and
$$
\mu(B)^{\alpha-1}w(B)^{\frac{1}{p}}\left(\int\limits_{B}w^{-p'/q}d\mu\right)^{\frac{1}{p'}}
\|\chi_{B}\| _{L_{w}^{q),\psi(x)}(X)} \leq
c\left(\eta_{B}w(B)\right)^{\frac{1}{p-\eta_{J}}}.
$$
For such  an $\eta_{B}$ we choose $\varepsilon_{B}$ so that
$$
\frac{1}{p-\eta_{B}}-\frac{1}{q-\varepsilon_{B}}=\alpha.
$$
Then $0<\varepsilon_{B}\leq q-1$ and
$$
\mu(B)^{\alpha-1}w(B)^{\frac{1}{p}-\frac{1}{p-\eta_{B}}}\eta_{B}^{-\frac{\theta}{p-\eta_{B}}}
\psi(\varepsilon_{B})^{\frac{1}{q-\varepsilon_{B}}}w(B)^{\frac{1}{q-\varepsilon_{B}}}
\left(\int\limits_{B}w^{-p'/q}d\mu\right)^{\frac{1}{p'}}\leq C. $$


Observe that since $\psi(t)\approx t^{ \theta (1+ \alpha q)}$ for small positive $t$, we have that

\begin{align*}
& \eta_{B}^{-\frac{\theta}{p-\eta_{B}}}\psi(\varepsilon_{B})^{\frac{1}{q-\varepsilon_{B}}}
 = \eta_{B}^{-\frac{\theta}{p-\eta_{B}}}\varphi\left(\varepsilon_{B}^{\theta}\right)^{\frac{1}{q-\varepsilon_{B}}}
\approx\eta_{B}^{-\frac{\theta}{p-\eta_{B}}}\varepsilon_{B}^{\frac{\theta(1+\alpha
q)}{q-\varepsilon_{B}}}\\ & =\left(\eta_{B}^{-\frac{1}{p-\eta_{B}}}\varepsilon_{B}^{\frac{1+\alpha
q}{q-\varepsilon_{B}}}\right)^{\theta}
\approx\left(\eta_{B}^{-\frac{1}{p-\eta_{B}}}\varphi(\varepsilon_{B})
^{\frac{1}{q-\varepsilon_{B}}}\right)^{\theta}=1
\end{align*}
and also,
$$
\frac{1}{p}-\frac{1}{p-\eta_{B}}+\frac{1}{q-\varepsilon_{B}}=\frac{1}{p}-\alpha=\frac{1}{q}.
$$
Finally, we have that
$$
\mu(B)^{\alpha-1} w(B)^{\frac{1}{q}}\left(\int\limits_{B}w^{-p'/q}\right)^{1/p'}\leq
C.
$$
The theorem has been proved.

\end{proof}

\subsection{Commutators}

We say that a function $b$ defined on $X$ belongs to $BMO$ if
$$ \|b\|_{BMO}= \sup_{B}\frac{1}{\mu(B)}\int\limits_B |b(x)- b_B|d\mu(x)  < \infty,$$
where $b_B =\frac{1}{\mu(B)} \int\limits_B b(y) d\mu(y)$.
\vskip+0.2cm

Let $b\in BMO(X)$, $m\in {\mathbb{N}}\cup \{0\}$ and let

$$ K_b^m f(x) = \int\limits_X [b(x)- b(y)]^{m} k(x,y) f(y) d\mu(y),   $$
where $k$ is the Calder\'on-Zygmund kernel.

It is known (see  \cite{PrSa}) that if $1<r<\infty$ and $w\in A_{\infty}$, then the one-weight inequality
$$\| K_b^m f\|_{L^r_w(X)} \leq C \|b \|_{BMO(X)}^{m} \| M^{m+1} f\|_{L^r_w(X)}, \;\; f\in {\mathcal{D}}(X),$$
holds, where $M^{m+1}$ is the the Hardy--Littlewood maximal operator iterated $m+1$ times.

\vskip+0.2cm

Based on extrapolation result in grand Lebesgue spaces we have

\begin{theorem}\label{Theorem 10}  Let $X$ be bounded and let $1<p,s<\infty$,  $\theta>0$. Then there is a positive constant $C$ such that for all $f\in {\mathcal{D}}(X)$ and all $w\in A_{p}(X)$,

$$\| K_b^m f\|_{L^{p), s, \theta}_w(X)} \leq C  \| M^{m+1} f\|_{L^{p),s,  \theta}_w(X)}, \;\; f\in {\mathcal{D}}(X).$$
\end{theorem}

\vskip+0.2cm



Further, for $b\in BMO(X)$, let

$$I^m_{\alpha, b}f(x) = \int\limits_X [b(x)- b(y)]^m K_{\alpha}(x,y) d\mu(y), \;\; 0<\alpha<1, $$

$${\mathcal{I}}^m_{\alpha, b}f(x) = \int\limits_X |b(x)- b(y)|^m K_{\alpha}(x,y) d\mu(y), \;\; 0<\alpha<1. $$

It is easy to see that, for $f\geq 0$, $|I^m_{\alpha, b}f(x)| \leq {\mathcal{I}}^m_{\alpha, b}f(x)$. In the same paper \cite{Bern} the authors showed that if $1<p< \infty$, $0<\alpha<1$,  $m\in {\mathbb{N}}\cup \{0\}$, $w\in A_{\infty}(X)$, $b\in BMO(X)$, then there is a constant $C\equiv C_{\alpha, m, p, \kappa, \mu}$ such that

$$ \int\limits_X |{\mathcal{I}}^m_{\alpha, b}f(x)|^p w(x) d\mu(x) \leq C N([w]_{A_{\infty}}) \| b\|_{BMO(X)}^{mp}\int\limits_X [ M_{\alpha} ( M^m f)(x) ]^p w(x) d\mu(x) $$
for some non-decreasing function $N$.

Based on this result and appropriate  extrapolation theorem we have the following statement:

\begin{theorem}\label{Theorem 11}  Let $1<p, s<\infty$, $m\in {\mathbb{N}}\cup \{0\}$  and let $\theta>0$. Suppose that  $X$ is bounded and that    $w\in A_{p}(X)$. Then there is a positive constant $C$ such that

$$\| {\mathcal{I}}^m_{\alpha, b} f\|_{L^{p), s, \theta}_w(X)} \leq C   \|M_{\alpha} ( M^m f)\|_{L^{p),s,  \theta}_w(X)}, \;\;\; f\in {\mathcal{D}}(X). $$
\end{theorem}

\begin{corollary} Under the conditions of Theorem \ref{Theorem 11} we have that  there is a positive constant $C$  such that for all $f\in {\mathcal{D}}(X)$,
$$ \| {\mathcal{I}}^m_{\alpha, b} f\|_{L^{p), s, \theta}_w(X)} \leq C   \|f\|_{L^{p),s,  \theta}_w(X)}. $$
\end{corollary}

\section{Further Remarks}


In this section we do some remarks regarding the results obtained in this paper.

\begin{remark}Let $1< p,s< \infty$. We can define new grand Lorentz space involving ''grandification'' of the second parameter $s$ in Lorentz space:
$f\in L^{p), s), \theta}_w$ if
$$ \| f\|_{L^{p), s), \theta}_w} =   \sup_{0<\varepsilon_1 < p-1,\; 0< \varepsilon_2<s-1 } \varepsilon_1^{\frac{\theta}{p-\varepsilon_1}} \|f\|_{
L^{p-\varepsilon_1, s- \varepsilon_2}_w } < \infty. $$

Analyzing the proofs of the main statements we can conclude that they are valid also for the spaces $L^{p), s), \theta}_w$.
\end{remark}

\begin{remark}

If we define grand Lorentz spaces with respect to the quasi-norms

$$ \| f\|_{ {\mathcal{L}}^{p), s), \theta}_w} =   \sup_{0< \varepsilon < \sigma} \varepsilon^{\frac{\theta}{p-\varepsilon}} \|f\|_{L^{p-\varepsilon, s}_w}$$

$$ \| f\|_{{\mathcal{L}}^{p), s, \theta}_w} =   \sup_{0<\varepsilon_1 < \sigma_1,\; 0< \varepsilon_2<\sigma_2 }
\varepsilon_1^{\frac{\theta}{p-\varepsilon_1}} \|f\|_{ L^{p-\varepsilon_1, s- \varepsilon_2}_w},$$
then the sufficiency part of  Theorems \ref{Theorem 5}--  \ref{Theorem 11} remain true even for unbounded $X$.
\end{remark}

\begin{remark}

Let $\varphi$ be a positive increasing function on $(0,p-1]$ such that $\lim_{x\to 0} \varphi (x) =0$.  Let us define the grand Lorentz space with respect to the quasi-norm:

$$ \| f\|_{ L^{p), s, \varphi}_w} =   \sup_{0< \varepsilon < \sigma} \varphi(\varepsilon)^{p-\varepsilon} \|f\|_{L^{p-\varepsilon, s}_w}.$$

Than again the results of this paper remains valid for such spaces.
\end{remark}

\section*{Acknowledgement}   The work was supported by the Shota Rustaveli National Science Foundation of Georgia (Project No. FR-18-2499).

\vspace{3.5 mm}

\noindent
Addresses:

\vskip+0.2cm

\noindent V.~Kokilashvili: Department of Mathematical Analysis,
A.~Razmadze Mathematical Institute, I.~Javakhishvili Tbilisi State
University,   Tamarashvili Str.~6, Tbilisi 0177, Georgia; and
International Black Sea University, 3~Agmashenebeli Ave., Tbilisi
0131,
Georgia.\\
\noindent E-mail: {\tt vakhtang.kokilashvili@tsu.ge}

\noindent A.~Meskhi: Department of Mathematical Analysis,
A.~Razmadze Mathematical Institute, I.~Javakhishvili Tbilisi State
University,  Tamarashvili Str.~6, Tbilisi 0177, Georgia; and
Department of Mathematics, Faculty of Informatics and Control
Systems,Georgian Technical University, 77, Kostava St.,  Tbilisi, Georgia.\\
\noindent E-mail: {\tt alexander.meskhi@tsu.ge}; {\tt a.meskhi@gtu.ge}


\begin{thebibliography}{99}


\bibitem{Ai} H. Aimar,  {\em Singular integrals and approximate identities on spaces of homogeneous type,}  Trans. Am. Math. Soc.  {\bf 292}   (1985), 135--153.



\bibitem{BeSh} C. Bennett and R.  Sharpley,  {\em Interpolation of operators.} Academic Press, London, 1988.


\bibitem{Bern} A. Bernardis,   S.  Hartzstein, and G.    Pradolini, {\em   Weighted inequalities for commutators of fractional
integrals on spaces of homogeneous type,}  J. Math. Anal. Appl. {\bf 322} (2006) 825--846.


\bibitem{CFK} C. Capone, A. Fiorenza, and G. E. Karadzhov, {\em Grand Orlicz spaces and global integrability of the Jacobian,}  Math. Scand., {\bf 102}(2008), No.1,  131--148.

    \bibitem{CR} R.-E. Castillo and H. Rafeiro, {\em An introductory course in Lebesgue spaces},  Canadian Mathematical Society, Springer, 2016.

\bibitem{ChHuKu}  H.-M. Chung, R. A. Hunt and D. S. Kurtz, {\em  The Hardy--Littlewood maximal function on $L(p,q)$ spaces with weight,}  Indiana Univ. Math. J.
 {\bf 31}(1982), No.1, 109-120.



\bibitem{CoWe} R.\,R.~Coifman and G.~Weiss,  \emph{Analyse harmonique non-commutative sur certains espaces homog\'enes},
Lecture Notes in Math., Vol. 242, Springer-Verlag, Berlin, 1971.

\bibitem{CFMP} D. Cruz-Uribe, A. Fiorenza, J.M. Martell, and C. Perez, {\em  The boundedness of classical operators on variable $L_p$ spaces,}
 Ann. Acad. Sci. Fenn. Math.  {\bf 31} (2006), No.1, 2006.

\bibitem{CMP} D. Cruz-Uribe, J. M. Martell and C. P\'erez,  Weights, extrapolation and the
theory of Rubio de Franc\'ia, volume 215 of Operator Theory: Advances and
Applications. Birkh\"auser/Springer Basel AG, Basel, 2011.


\bibitem{CGMP}  G. P. Curbera, J. Garc\'ia-Cuerva, J. M. Martell and C. P\'erez, {\em Extrapolation with weights, rearrangement-invariant function spaces, modular inequalities and applications to singular integrals,}  Adv. Math.  {\bf 203} (2006), No. 1, 256--318.



\bibitem{DGPP} O. Dragi\v{c}evi\'c, L. Grafakos, M. C. Pereyra and S.  Petermichl,  {\em Extrapolation and sharp norm estimates for
classical operators on weighted Lebesgue spaces, }  Publ. Mat. {\bf 49} (2005),  No. 1,    73--91.


\bibitem{Doan1} J.~Duoandikoetxea, Extrapolation of weights revisited: New proofs and sharp bounds,
 J.  Funct. Anal. {\bf 260} (2011) 1886--1901.

\bibitem{EKM} D. E. Edmunds, V. Kokilashvili and A. Meskhi, {\em Bounded and compact integral operators}, Kluwer Academic Publishers, Dordrecht,
Boston, London, 2002.

\bibitem{FiKaZAA} A. Fiorenza and G. E. Karadzhov, {\em Grand and small Lebesgue spaces
and their analogs},  Zeits. Anal. Anwen. {\bf 23} (2004), No. 4, 657--681.

\bibitem{F} N. Fujii, {\em Weighted bounded mean oscillation and singular integrals}, Math. Japan. {\bf 22} (1977/78), No. 5, 529--534.


\bibitem{GeKo} I. Genebashvili   and V. Kokilashvili, {\em Weighted norm inequalities for fractional functions and integrals
dehned on homogeneous type space,}  Proc. A. Razmadze Math. Inst. Georgian Acad. Sci.  {\bf 106} (1993), 63--76.



\bibitem{Gra} L. Grafakos,  {\em Classical Fourier analysis}.  Third Edition. Graduate Texts in Mathematics,  249, New York: Springer, 2014.



\bibitem{GIS} L. Greco, T. Iwaniec and C. Sbordone, {\em Inverting the $p$-harmonic operator,}  Manuscripta Math. {\bf 92} (1997), 249--258.



\bibitem{HMS} E. Harboure, R. Mac\'ias, C. Segovia,  {\em Extrapolation results for classes of weights,}   Amer. J. Math. {\bf 110} (1988),   383--397.


\bibitem{Ho} K.-P. Ho, {\em Strong maximal operator on mixed-norm spaces},  Ann. Univ. Ferrara  {\bf 62} (2016), No 2,  275--291.


\bibitem{Hu} R. A.  Hunt,  {\em On $L(p, q)$ spaces},  Enseign. Math.   {\bf 12}  (1966), 249--274.


\bibitem{HuPeRe} T. P. Hyt\"onen, C. P\'erez, and E. Rela,  {\em Sharp reverse H\"older property for $A_{\infty}$ weights on spaces of homogeneous type,} J. Funct. Anal.  {\bf 263(12)} (2012), 3883--3899.


    \bibitem{Hru} S. Hru\v{s}\v{c}ev, {\em A description of weights satisfying the $A_{\infty}$ condition of Muckenhoupt},  Proc. Amer. Math. Soc. {\bf 90} (1984), 253--257.


        \bibitem{IwSb} T. Iwaniec and C.  Sbordone, {\em On the integrability of the Jacobian under minimal hypotheses,}
Arch. Ration. Mech. Anal. {\bf 119} (1992),  129--143.



\bibitem{JK}     P. Jain and S. Kumari,  {\em On grand Lorentz spaces and the maximal operator,}  Georgian Math. J.  {\bf 19} (2012), No.2, 235--246.

\bibitem{KoMANA} V. Kokilashvili, {\em Weighted estimates for maximal functions and fractional integrals in Lorentz spaces,}  Math.Nachr.  {\bf 133} (1987), 33--42.


\bibitem{KK}  V. Kokilashvili and M. Krbec, {\em Weighted inequalities in Lorentz and Orlicz spaces},  World Scientific, Singapore, 1991.


\bibitem{KoMeGMJ} V.  Kokilashvili  and A.  Meskhi, {\em A note on the boundedness of the Hilbert transform in weighted grand Lebesgue spaces,}   Georgian Math. J. {\bf  16} (2009), No.3, 547--551.

\bibitem{KMMIAN}  V.  Kokilashvili  and A.  Meskhi, {\em Weighted extrapolation in Iwaniec--Sbordone spaces.
Applications to Integral Operators and Approximation Theory,}   Proceedings of the Steklov Institute of Mathematics  {\bf 293} (2016),  161--185.


\bibitem{KoMePositivity} V. Kokilashvili and A. Meskhi, {\em Extrapolation results in grand Lebesgue spaces defined on product sets},  Positivity
 {\bf 22} (2018), No.4,  1143--1163.

\bibitem{KMRS2} V. Kokilashvili, A. Meskhi, H. Rafeiro and S. Samko,  {\em Integral operators in non-standard
function spaces: Variable exponent H\"older, Morrey-Campanato and grand spaces}, Volume 2,   Birkh\"auser/Springer, Heidelberg, 2016.


\bibitem{Ku}  D. Kurtz, {\em Classical operators on mixed-normed spaces with product weights},  Rocky Mountain Journal of Mathematics  {\bf 37} (2007), No.1, 269--283.


\bibitem{LMPT}
M. T. Lacey,  K. Moen,  C.  Perez, R. H. Torres,  {\em Sharp weighted bounds for fractional integral operators},   J.  Funct. Anal.  {\bf 259} (2010),   1073--1097.

\bibitem{MS} R.\,A.~Mac\'ias and C.~Segovia, {\em Lipschitz functions on spaces of homogeneous type}, Adv. Math. {\bf 33} (1979), 257--270.


\bibitem{MeJMS} A. Meskhi, {\em Weighted criteria for the Hardy transform under the $B_p$ condition in grand Lebesgue spaces and some applications,}  J. Math. Sci. (N.Y.),  {\bf 178} (2011), No. 6, 622--636.

\bibitem{Mes} A. Meskhi,	{\em Criteria for the boundedness  of potential operators in grand Lebesgue spaces},  Proc. A. Razmadze Math. Inst.  {\bf 169} (2015), 119-132.


\bibitem{PrSa} G. Pradolini and O.  Salinas, {\em Commutators of singular integrals on spaces of homogeneous type},  Czechoslovak Math. J.   {\bf 57} (2007), 75--93.

\bibitem{Rubio} J. L. Rubio de Francia, {\em Factorization and extrapolation of weights},  Bull. Amer. Math. Soc. (N.S.) {\bf 7}  (1982) 393--395.


\bibitem{Rubio1} J. L. Rubio de Francia, {\em Factorization theory and $A_p$ weights},  Amer. J. Math.   {\bf 106} (1984),  533--547.

\bibitem{SS}  I. M. Stein and G. Weiss, {\em Introduction to Fourier analysis on Euclidean spaces}, Princeton,
New Jersey, Princeton University Press, 1971.

\bibitem{StTo} J.\,O.~Str\"omberg and A.~Torchinsky, \emph{Weighted {H}ardy spaces},  Lecture  Notes in Math. Vol.  1381, Springer Verlag, Berlin,  1989.




\end{thebibliography}
\end{document}